\documentclass{amsart}
\usepackage{hyperref}
\usepackage{url}

\usepackage{amsmath,amssymb,amsthm}
\usepackage{fancyvrb}
\usepackage{enumerate} 
\usepackage{cases} 
\usepackage{color} 
\usepackage{graphicx}
\usepackage{bm}
\usepackage{geometry}

\newcommand{\subf}[2]{%
	{\small\begin{tabular}[t]{@{}c@{}}
			#1\\#2
	\end{tabular}}%
}

\numberwithin{equation}{section} \numberwithin{figure}{section}

{\theoremstyle{remark} \newtheorem{remark}{Remark}}
\newtheorem{definition}{Definition}[section]

\newtheorem{theorem}{Theorem}[section]

\newtheorem{lemma}{Lemma}[section]
{\theoremstyle{remark} \newtheorem{example}{Example}[section]}

\newcommand\pp{\partial}
\newcommand{\eps}{\varepsilon}

\newcommand{\calA}{\mathcal{A}}
\newcommand{\calC}{\mathcal{C}}

\newcommand{\calI}{\mathcal{I}}
\newcommand{\calL}{\mathcal{L}}

\newcommand{\calK}{\mathcal{K}}
\newcommand{\calD}{\mathcal{D}}
\newcommand{\calN}{\mathcal{N}}
\newcommand{\dD}{d_\calD}
\newcommand{\dN}{d_\calN}
\newcommand{\gD}{g_\calD}
\newcommand{\gN}{g_\calN}
\newcommand{\GD}{G_\calD}
\newcommand{\GN}{G_\calN}
\newcommand{\GammaD}{\Gamma_\calD}
\newcommand{\GammaN}{\Gamma_\calN}

\newcommand{\N}{\mathbb{N}}
\newcommand{\R}{\mathbb{R}}
\newcommand{\Rd}{{\mathbb{R}^d}}

\newcommand{\vA}{{{\bm A}}}

\newcommand{\vb}{{{\bm b}}}
\newcommand{\ve}{{{\bm e}}}
\newcommand{\vm}{{{\bm m}}}
\newcommand{\vn}{{{\bm n}}}

\newcommand{\vq}{{{\bm q}}}
\newcommand{\vt}{{{\bm t}}}

\newcommand{\vy}{{{\bm y}}}
\newcommand{\vD}{{{\bm D}}}

\newcommand{\vbeta}{{{\bm \beta}}}

\newcommand{\vnu}{{{\bm \nu}}}
\newcommand{\vphi}{{{\bm \phi}}}
\newcommand{\vphii}{{{\bm \varphi}}}
\newcommand{\vpsi}{{{\bm \psi}}}
\newcommand{\vTheta}{{{\bm \Theta}}}

\newcommand\restr[2]{{
  \left.\kern-\nulldelimiterspace 
  #1 
  \vphantom{\big|} 
  \right|_{#2} 
  }}

\DeclareMathOperator\Div{div}

\begin{document}

\title[A deep least-squares method]{A deep first-order system least squares method for solving elliptic PDEs}

\author[F.M.~Bersetche]{Francisco M.~Bersetche}
\address[F.M.~Bersetche]{Departamento de Matem\'atica, Universidad de Buenos Aires, Buenos Aires, Argentina}
\email{fbersetche@dm.uba.ar}
\thanks{ FMB has been supported in part by a PEDECIBA postdoctoral fellowship and the ANPCyT grant PICT 2018-3017.}

\author[J.P.~Borthagaray]{Juan Pablo~Borthagaray}
\address[J.P.~Borthagaray]{Centro de Matem\'atica, Universidad de la Rep\'ublica, Montevideo, Uruguay}
\email{jpb@cmat.edu.uy}

\begin{abstract}
We propose a First-Order System Least Squares (FOSLS) method based on deep-learning for numerically solving second-order elliptic PDEs. The method we propose is capable of dealing with either variational and non-variational problems, and because of its meshless nature, it can also deal with problems posed in high-dimensional domains. We prove the $\Gamma$-convergence of the neural network approximation towards the solution of the continuous problem, and extend the convergence proof to some well-known related methods. Finally, we present several numerical examples illustrating the performance of our discretization.
\end{abstract}
 
\maketitle

\section{Introduction} \label{sec:introduction}

Approximate solution of PDEs using machine learning techniques has been considered in various forms in the past thirty years. For instance,  \cite{lagaris1998,lagaris2000,lee1990,malek2006} propose to use neural networks to solve PDEs and
ODEs. These articles compute neural network solutions by using an a priori fixed mesh. In recent years, there has been an incipient development of mesh-free numerical methods to solve PDEs by using neural networks. Although the approaches have been diverse, most of these algorithms aim to train a neural network to approximate the unknown function, forcing the fulfillment of the PDE and its boundary conditions through a suitable loss functional. In this regard, among other works, let us mention \cite{E18, DGM18, He20, Wang20, Zang20, PINN, FNM, liu1, liu2}.

Deep neural networks are not necessarily suitable for solving PDEs in low dimensions, where they may be outperformed by classical methods specifically tailored for the problems under consideration. However, neural network methods have proven to be effective in some circumstances where the application of classical methods becomes impractical. Such is the case of high-dimensional PDEs. We refer to \cite{wojtowytsch2020can, weinan2022some} for discussion about the suitability of shallow neural networks for solving high-dimensional PDEs.
     
The method we propose in this work aims
to overcome some disadvantages of the algorithms available in the literature.
Using a first-order formulation we are able to avoid the computation of second-order derivatives in cost functionals, thereby saving a significant computational cost in high dimensions. Avoiding second-order derivatives also allows us to use linear activation functions and, a priori, gives us the possibility of approximating weak solutions. On the other hand, counting on explicit representations of the gradients simplifies the strong imposition of
 Neumann-type boundary conditions. Namely, we can impose boundary conditions without adding penalty terms in the loss function.  This results in a reduction in training time. First order formulations have also been used in \cite{liu1,liu2}.

Let $\Omega \subset \Rd$ be an open domain. In this work, we shall make use of the spaces
\[ \begin{aligned}
&H^1(\Omega) = \{ v \in  L^2(\Omega) \colon \nabla v \in L^2(\Omega)\}, \\
& H(\Div; \Omega ) = \{ \vpsi \in  [L^2(\Omega)]^d \colon \Div \vpsi \in L^2(\Omega)\}.
\end{aligned}
\]
We assume there exists a disjoint partition $\pp \Omega = \GammaD \cup \GammaN$, with $|\GammaD| > 0$, and let $\vnu$ denote the outward normal to $\Omega$. Given sufficiently regular functions $f, \gD, \gN$, we aim to solve the problem 
\begin{equation}\label{eq:problem}
\left\lbrace
\begin{aligned}
- \mbox{div}(\vA \nabla u) + Bu & = f  &\mbox{ in } \Omega, \\
u & = \gD  &\mbox{ on } \GammaD, \\
\vA \nabla u  \cdot \vnu & = \gN &\mbox{ on } \GammaN, \\
\end{aligned}
\right.
\end{equation}
where we assume $\vA \in [L^{\infty}(\Omega)]^{d\times d}$ is a.e. symmetric and uniformly positive definite: there exist constants $\lambda, \Lambda$ such that
\[
0 < \lambda \le \lambda_{\min}(\vA(x)) \le \lambda_{\max}(\vA(x)) \le \Lambda, \quad \mbox{for a.e. } x \in  \Omega,
\]
where $\lambda_{\min}(\vA(\cdot))$ (resp. $\lambda_{\max}(\vA(\cdot))$) denotes the minimum (resp. maximum) eigenvalue of $\vA(\cdot)$.

We assume the linear operator $B \colon H^1(\Omega) \to L^2(\Omega)$ in \eqref{eq:problem} satisfies
\begin{equation*} \label{eq:hyp-B}
\| B v \|_{L^2(\Omega)} \le C \| \nabla v \|_{L^2(\Omega)}  \quad \forall v \in H^1(\Omega) \mbox{ such that } v = 0 \mbox{ on } \GammaD.
\end{equation*}
Examples satisfying this condition include $Bv = \Div(\vbeta v)$, with $\vbeta \in [W^{1,\infty}(\Omega)]^d$, 
and $Bv = \vbeta \cdot \nabla v + \gamma v$ for some $\vbeta \in [L^{\infty}(\Omega)]^d$, $\gamma \in L^{\infty}(\Omega)$. We thus remark that \eqref{eq:problem} can accommodate, for example, stationary convection-reaction-diffusion problems. Following \cite{Cai94}, we require problem \eqref{eq:problem} to be invertible in $H^1(\Omega)$, namely, that for every $f \in H^{-1}(\Omega)$ there exists a weak solution $u \in H^1(\Omega)$ with $u=0$ on $\GammaD$.


We introduce the flux variable $\vphi = \vA \nabla u$ and rewrite \eqref{eq:problem} as a first-order system:
\begin{equation}\label{eq:FOS}
\left\lbrace
\begin{aligned}
 \vphi - \vA \nabla u & = 0 & \mbox{ in } \Omega, \\
 - \Div \vphi + B u - f & = 0 & \mbox{ in } \Omega, \\
u & = \gD & \mbox{ on } \GammaD, \\
\vphi \cdot \vnu & = \gN & \mbox{ on } \GammaN. \\
\end{aligned} 
\right.
\end{equation}
Our approach is based on seeking minimizers of the loss function 
\begin{equation} \label{eq:LS-loss}
\calL (u, \vphi) := \|  \vphi - \vA \nabla u \|_{L^2(\Omega)}^2 + \| \Div \vphi - B u + f \|_{L^2(\Omega)}^2
\end{equation}
on a suitable set of admissible functions
\[
\calA := \{ \vq = (u, \vphi) \in H^1(\Omega) \times H(\Div; \Omega) \colon u = \gD \mbox{ on } \GammaD, \ \vphi \cdot \vnu = \gN \mbox{ on } \GammaN \}.
\]
Clearly, if \eqref{eq:problem} has a unique solution $u \in H^1(\Omega)$, then the unique minimizer of $\calL$ in $\calA$ is $\vq :=~(u, \vA \nabla u)$.
Our goal is to compute approximations to such a minimizer within a suitable 
space $\calA_m \subset \calA$. Particularly, in our method we consider a space $\calA_m$ composed of neural networks with a fixed architecture and parameters $\vTheta \in \R^m$. Some efforts in this direction include the deep FOSLS method from \cite{Cai20} and the deep mixed residual method proposed in \cite{Lyu20}. 
Reference \cite{Cai20} proposes the use of a partition of $\Omega$ and a mid-point quadrature rule for the evaluation of the discrete loss functional; instead, our algorithm is meshfree and uses random quadrature points. 
In more recent work by three of the authors of that work \cite{liu1}, the use of Monte Carlo integration is discussed albeit not pursued in detail.
Such an approach yields a significant advantage in high-dimensional problems. Our method can be understood in the setting of the mixed residual methods in \cite{Lyu20}. However, a significant difference between our work and \cite{Lyu20} is that here we propose a strong imposition of the boundary conditions instead of the inclusion of penalization terms in the loss functional. We pre-train neural networks to accommodate boundary data, which results in a reduction in the number of iterations required in the solution of the PDE \cite{Berg18, Sheng20}.

The error in the approximation of continuous functionals with their discrete counterparts is usually not taken into account in numerical methods based on neural networks available in the literature; some recent efforts in this direction include \cite{shin2020,hong2021priori,zerbinati2022pinns}, where convergence rates are proved for a certain class of elliptic functionals under strong regularity conditions on the solution of the continuous problem.
In other words, the focus is generally on the convergence of the minimizers of functionals such as $\calL$ in \eqref{eq:LS-loss} over certain neural network spaces towards the minimizer of the same functional at the continuous level. However, in practice one does not compute $\calL$ exactly but rather approximates it by means of quadrature rules. Let us call $\calL_N$ such an approximation to the functional $\calL$, where $N$ is, for example, the number of quadrature points. The computation of $ \calL_N $ instead of $ \calL $ can introduce important changes in the nature of the minimization problem, such as the loss of convexity of the associated functional \cite{E18}. A major contribution of this work is to present a convergence analysis that considers the discretization of the functional $ \calL $. Specifically, we prove the almost-sure $\Gamma$-convergence of the discrete loss functions towards the continuous one. As stated in Theorem \ref{teo:conv}, this implies the almost-sure convergence of the solutions computed numerically to the solution of the continuous problem.

The techniques we develop for this purpose are not only valid for the method we propose, and we generalize and apply them to the convergence analysis of a broad class of methods, including the Deep Ritz \cite{E18} and the Deep Galerkin \cite{DGM18} Methods (DRM and DGM, respectively; see Remarks \ref{rem:DRM} and \ref{rem:DGM}).

\subsection*{Organization of the paper.} The rest of the paper is organized as follows.
Section \ref{sec:description} describes the method we propose for dealing with \eqref{eq:FOS}, including the treatment of Dirichlet and Neumann boundary conditions in strong form, and discusses some aspects pertaining to its implementation. We perform a convergence analysis for our method in Section \ref{sec:analysis}. This analysis takes into account the approximation of the loss functional by means of Monte Carlo integration, and establishes the convergence of the discrete minimization problem towards the continuous one in the sense of almost sure $\Gamma$-convergence. Section \ref{sec:general} generalizes the analysis to include some other well-known methods, thereby establishing their convergence as well. We illustrate the performance of our method through computational examples in Section \ref{sec:numerical}, and provide some concluding remarks in Section \ref{sec:conclusion}.

\section{Description of the method} \label{sec:description}

The goal of the method we propose is to approximate the unique minimizer $(u, \vphi)$ of the functional in \eqref{eq:LS-loss}. A natural first approach would consist in seeking a set of parameters $\vTheta_0 \in \R^m$ such that
\[
\calL (u_{ \vTheta_0}, \vphi_{ \vTheta_0 }) = \min_{\vTheta \in \R^m} \calL (u_{\vTheta}, \vphi_{\vTheta}),
\]
with the functions $(u_{\vTheta}, \vphi_{\vTheta})$ belonging to a suitable neural network space. The use of neural networks in this setting has the advantage that one can easily implement meshfree methods by randomly sampling collocation points (see \cite{DGM18,he2022relu,siegel2020high,siegel2021sharp}, for example), and thereby be able to deal with high-dimensional problems, where most classical numerical PDE methods become unfeasible.

The enforcement of boundary conditions is a non-trivial aspect to take into account in this approach. A typical way to tackle this issue is to incorporate boundary conditions by adding a penalization term \cite{ E18, DGM18, Zang20}. However, in practice it is observed that enforcing discrete functions to satisfy the boundary conditions gives rise to a faster training process \cite{Berg18,Sheng20}.
We shall first create suitable auxiliary functions with the purpose of imposing the boundary conditions in a strong fashion. In this way, we ensure ($u_{\vTheta},\vphi_{\vTheta}) \in \calA_m \subset \calA$, for all $\vTheta \in \R^m$. Then, the optimization procedure consists of sampling $N$  points $\{x_k \}_{k=1}^N \subset \Omega$ uniformly, and approximating
$\calL (u_{ \vTheta}, \vphi_{ \vTheta}) \approx \calL_N (u_{ \vTheta}, \vphi_{ \vTheta}),$ 
at every step of a gradient descent algorithm, with $\calL_N$ defined as
\begin{equation}\label{eq:discrete_cost}
	\calL_N (u, \vphi) :=  \frac{|\Omega|}{N} 
	\sum_{k=1}^N  \big( \vphi(x_k) - \vA \nabla u(x_k) \big)^2 + \big(\Div\vphi(x_k) - B u(x_k) + f(x_k) \big)^2.
\end{equation}        
We expose the details below.

\subsection{Strong imposition of boundary conditions}\label{subsec:strong_b}
We follow the ideas from \cite{Berg18} about the imposition of Dirichlet boundary conditions, and extend the approach to include Neumann boundary conditions. Instead of trying to compute either $u$ or $\vphi$ directly and incorporate the boundary conditions by a penalization term, we shall enforce them in the construction of the neural network approximations. For that purpose, we make use of the following notion.

\begin{definition}[smooth distance function] \label{def:distance-function}
Let $\Gamma_* \subset \overline \Omega$ be a closed set. We say that a Lipschitz continuous function $d_* \colon \Omega \to \R$ is a {\em smooth distance function} if it satisfies $d_* \ge 0$ and $d_* (x) = 0$ if and only if $x \in \Gamma_*$.
\end{definition}

We briefly comment on the use of smooth distance functions in the strong imposition of Dirichlet and Neumann boundary conditions. In the computation of $u$ in \eqref{eq:LS-loss}, we restrict the class of functions to be
\begin{equation} \label{eq:def-u}
u (x) := \GD(x) + \dD(x) \, v(x),
\end{equation}
where the unknown is the function $v \colon \Omega \to \R$, $\GD$ is a lifting of the Dirichlet datum, and $\dD$ is a smooth distance function to $\GammaD$.

In a similar fashion, we can incorporate normal boundary conditions on the flux variable $\vphi$ in a strong way. We first construct a vector field $\vn \colon \Omega \to \R^d$ such that $\vn |_{\GammaN} = \vnu$ and $| \vn (x) | = 1$ for a.e. $x \in \Omega$, and consider
\begin{equation} \label{eq:def-vphi}
\vphi (x) := \vpsi(x) + \left(\GN (x) - \frac{  \vpsi(x) \cdot \vn(x) }{1 + \dN(x)} \right) \vn(x).
\end{equation}
Above, $\GN$ is a lifting of the Neumann boundary condition, $\dN$ is a smooth distance function to $\GammaN$, and the unknown is the function $\vpsi \colon \Omega \to \R^{d}$. By its definition, the function $\vphi$ satisfies the boundary condition $\vphi \cdot \vnu = \gN$ at $\GammaN$. We remark that we do not require any smoothness on $\vn$: in particular this field may be discontinuous at some points in the domain.

Therefore, in the construction of approximate solutions we shall first compute the vector field $\vn$ and the scalar functions $\dD$, $\dN$, $\GD$, $\GN$. Then, we seek $\vy = (v, \vpsi)$ such that the corresponding pair $(u,\vphi)$, given by \eqref{eq:def-u} and \eqref{eq:def-vphi}, minimizes the loss function $\calL$. The computation of the auxiliary functions $\vn$, $\dD$, $\dN$, $\GD$, $\GN$ typically requires fewer degrees of freedom and iterations than the computation of $(v, \vpsi)$, depending on the complexity of the domain or the boundary data. Consequently, we shall frequently use a simpler architecture to represent them. Below, we give details on the computation of the auxiliary functions.

\subsubsection{Computation of smooth distance functions}
Loosely, for $* \in \{\calD,\calN\}$, a smooth distance function to $\Gamma_*$ is a function $d_* \colon \Omega \to [0,\infty)$ that approximates the distance to $\Gamma_*$, cf. Definition \ref{def:distance-function}. To construct such functions, we first randomly choose $N_d$ points $\{x_{i}\}_{i=1}^{N_d} \subset \Omega$ (the same set of points can be used for either $* = \calD$ and $* = \calN$) and compute
\[
d^{(*)}(x_i) \approx \mbox{dist} (x_i, \Gamma_*).
\]
This can be done by choosing points on $\Gamma_*$ and using efficient nearest-neighbor search strategies. Once we have computed the quantities $\{d^{(*)}(x_i)\}_{i=1}^{N_d}$, we train a neural network for $d_*$ by using the cost function 
\[
\calL_{*}(d) = \frac{1}{N_d} \sum_{i=1}^N |d(x_i) - d^{(*)}(x_i)|^2 + \frac{1}{N_{d,*}} \sum_{i=1}^{N_{d,*}} |d(x_{*,i})|^2,
\]
where $\{x_{*,i}\}_{i=1}^{N_{d,*}}$ is a random batch of points on $\Gamma_*$. 

In the setting of $\dD$ and $\dN$, we use neural networks with a single hidden layer and significantly less parameters than the networks employed in the PDE resolution.

\subsubsection{Boundary data liftings and normal field} \label{sec:boundary-data}
We approximate liftings of the boundary data to $\Omega$ by {\em smooth liftings} \cite{Berg18}: in either \eqref{eq:def-u} and \eqref{eq:def-vphi}, we require $\GD$ and $\GN$ to coincide with $\gD$ on $\GammaD$ and with $\gN$ on $\GammaN$, respectively, and to be smooth enough so that we can apply the differential operator to them pointwise. A natural way to enforce the former is to set the $L^2$-norms of the discrepancies on the corresponding boundary subsets as loss functions,
namely
\[
\calL_\calD (G) = \| G - \gD \|_{L^2(\GammaD)}^2 , \quad \calL_\calN (G) = \| G - \gN \|_{L^2(\GammaN)}^2.
\]

In practice, we consider sets of boundary nodes $\{ z^\calD_i \}_{i=1}^{M_\calD} \subset \GammaD,$ $\{ z^\calN_i \}_{i=1}^{M_\calN} \subset \GammaN$ and define the quadratic cost functionals
\[
\calL_\calD (G) = \frac1{M_\calD}\sum_{i=1}^{M_\calD} | G(z^\calD_i) - \gD(z^\calD_i) |^2,
\qquad
\calL_\calN (G) = \frac1{M_\calN}\sum_{i=1}^{M_\calN} | G(z^\calN_i) - \gN(z^\calN_i) |^2.
\]
In the same fashion as for the smooth distance functions, we consider neural networks with a single hidden layer to compute the functions $\dD$ and $\dN$.

Analogously, for the computation of the vector field $\vn$ we start from the loss function
\[
\calL_{\vn} (\vm) = \| \vm - \vnu \|_{L^2(\GammaN)}^2 + \| |\vm|^2 - 1 \|_{L^2(\Omega)}^2 ,
\]
consider a set of randomly selected points $\{z^{\calN, \vn}_{i}\}_{i=1}^{M_{\calN, \vn}} \subset \GammaN$ and $\{z^{\vn}_i\}_{i=1}^{M_{\vn}} \subset \Omega$, and minimize the cost functional 
\[
\calL_{\vn} (\vm) = \frac{1}{M_{\calN, \vn}} \sum_{i=1}^{M_{\calN, \vn}} | \vm(z^{\calN, \vn}_{i}) - \vnu |^2 + \frac1{M_{\vn}}\sum_{i=1}^{M_{\vn}} | |\vm(z^{\vn}_i)|^2 - 1 |^2.
\]
We point out that, in practice, the set of auxiliary points $\{z^{\calN, \vn}_{i}\}_{i=1}^{M_{\calN, \vn}}$ can be the same as the set $\{ z^\calN_i \}_{i=1}^{M_\calN}$ used in the approximation of $\calL_\calN$.

\subsection{Computational aspects}
Once we have built the auxiliary functions, we proceed to compute $u$ and $\vphi$. For this purpose, we consider a set of random points $\{x_k\}_{k=1}^N \subset \Omega$, and seek to minimize the cost functional 
\begin{equation}\label{eq:loss_discrete}
	\calL_N (u, \vphi) :=  \frac{|\Omega|}{N} 
	\sum_{k=1}^N  \big( \vphi(x_k) - \vA \nabla u(x_k) \big)^2 + \big(\Div\vphi(x_k) - B u(x_k) + f(x_k) \big)^2.
	\end{equation}
From the construction of $u$ and $\vphi$ (see \eqref{eq:def-u} and \eqref{eq:def-vphi}), the trainable parameters $ \vTheta $ arise in the computation of the auxiliary functions $v$ and $\vpsi$.

In broad terms, the method we propose can be summarized as follows: 
\begin{itemize}
 \item {\bf Stage 1:} Train auxiliary functions $\dD$, $\dN$, $\GD$, $\GN$, and $\vn$.
\item {\bf Stage 2:} Until some stop criterion is reached, do:
\begin{itemize}
	\item Select random points $\{x_k\}_{k=1}^N \subset \Omega$.
	\item For some learning rate $\ell$, do: $$\vTheta = \vTheta - \ell \nabla_{\vTheta} \calL_N (u_\vTheta, \vphi_\vTheta).$$
	\item Update learning rate.
\end{itemize}
\end{itemize}

The computation of $\calL_N (u, \vphi)$ requires computing the derivatives of $u$ and $\vphi$ with respect to the input variables, evaluated at $\{x_k\}_{k=1}^N$. Since we constructed our auxiliary functions as neural networks, it is possible to compute efficiently these derivatives by means of the Back-Propagation algorithm. Packages like TensorFlow allow this kind of computation.
  
Additionally, our least-squares loss function \eqref{eq:loss_discrete} only involves first-order derivatives in space. We discretize such derivatives by using finite-difference quotients. 
Namely, for any function $\vphii \colon \Rd \to \R^n$ we let $h > 0$ be a fixed constant and consider the second-order (with respect to $h$) formula
\[
\partial_i \vphii (x_k) \simeq \frac{\vphii  (x_k + h\ve_i) - \vphii (x_k - h\ve_i)}{2h},
\]
where $\ve_i \in \R^d$ is the $i$-th canonical basis vector in $\Rd$. We employ this formula for the approximation of $\nabla u$, $\Div \vphi$ and the first-order derivatives involved in $B$.

For the numerical examples we implemented our algorithm by using PyTorch and discretizing the differential operators by means of finite differences. We typically use about 10,000 steps of gradient descent, sampling between 1,000 and 5,000 random points in $\Omega$ at each step. A step-type decrease in the learning rate showed good results in practice. In particular, we start from a learning rate $\ell = 10 ^{-2}$, which we 
halve every 1,000-2,500 gradient descent steps.
No particular type of architecture was chosen for the functions involved. We use three-layer neural networks with linear activation function (ReLU) for the auxiliary functions, and five-layer networks for the main variables $v$ and $\vpsi$. The ADAM \cite{ADAM} optimization algorithm showed good results in numerical experiments. Further details about the implementation of the method can be found in Section \ref{sec:numerical}.

Regarding the training of auxiliary functions $\dD$ and $\dN$, the following procedure showed good results in practice: 
\begin{itemize}
    \item Select $N_d$ random points  $\{x_k\}_{k=1}^{N_d} \subset \Omega$.
    
    \item Initialize a vector $\vD$ as $\vD_i = \infty$ for $i = 1,...,N_d$.
	
	\item Until some stop criterion is reached, do:
	
	\begin{itemize}
		\item Select $M_{*}$ random points $\{z^*_i\}_{i=1}^{M_*} \subset \Gamma_*$.
		\item Update $\vD$ as: $\vD_k =\min\{ \min_{i= 1, \ldots M_{*}}{|x_k - z^*_i|} , \vD_k\}$
		\item Define the loss function: 
		$$\calL_{*}(d_*) = \frac{1}{N_d} \sum_{k=1}^{N_d} |d_*(x_k) - \vD_k|^2 + \frac{1}{M_*} \sum_{i=1}^{M_{*}} |d(z_i^*)|^2.$$
		\item For some learning rate $\ell$, do: $$\vTheta_{d_*} = \vTheta_{d_*} - \ell \nabla_{\vTheta_{d_*}} \calL_* (d_*).$$
		\item Update learning rate.
	\end{itemize}
Here $* \in \{D,N\}$, and $\vTheta_{d_*}$ denotes the trainable parameters of $d_*$.
\end{itemize}
 
\section{Analysis of the method}  \label{sec:analysis}

In this section, we prove the convergence of our method by using two main ingredients. First, we put the discretization in a $\Gamma$-convergence framework. More precisely, the sequence of functionals we consider is related to the use of meshfree methods in the computation of a regularized version of the discrete loss functional $\R^m \ni \vTheta \mapsto \calL (u_\vTheta, \vphi_\vTheta)$; see Theorem \ref{teo:gamma_conv} below. Second, we exploit the coercivity of the least-squares functional and approximation properties of neural networks to conclude that the sequence of minimizers of the regularized discrete loss functionals converges to the solution of \eqref{eq:problem} as the number of neural network parameters $m \to \infty$.

For the sake of simplicity, we consider problem \eqref{eq:FOS} with $\gD =  \gN = 0$. Otherwise, one could consider $\GD$ and $\GN$ such that $\GD = \gD$ on $\GammaD$ and $\GN = \gN$ on $\GammaN$, a smooth normal field $\vn$ such that $\vn = \vnu$ on $\GammaN$, and then the auxiliary functions $u_0 = u - \GD$ and $\vphi_0 = \vphi - \GN \vn$ would solve the first-order system
\[
\left\lbrace
\begin{array}{rll}
 \vphi_0 - \vA \nabla u_0 = &  \vA \nabla \GD - \GN \vn  & \mbox{in } \Omega, \\
 - \Div(\vphi_0) + B u_0 = & f + \Div(\GN \vn) - B \GD  & \mbox{in } \Omega, \\
u_0 = & 0 & \mbox{on } \GammaD, \\
\vphi_0 \cdot \vnu = & 0 & \mbox{on } \GammaN. \\
\end{array}
\right.
\]
Naturally, the solution to this system corresponds to the minimum of the least-squares functional
\[
(u, \vphi) \mapsto \|  \vphi - \vA \nabla u + \widetilde{g} \|_{L^2(\Omega)}^2 + \| \Div(\vphi) - B u + \widetilde{f} \|_{L^2(\Omega)}^2,
\]
with $\widetilde{g} = - \vA \nabla \GD + \GN \vn $ and $\widetilde{f} = f + \Div(\GN \vn) - B \GD$. This functional can be dealt with by using the same tools as for \eqref{eq:LS-loss}, the only difference being the presence of the zero-order correction term $\widetilde{g}$ in the first $L^2$-norm.

In the following proof of convergence, we restrict ourselves to one hidden layer neural networks with $n$ neurons. We define the set of discrete functions
\begin{equation*}
	\calC_m := \Big\{ (v_\vTheta,\vpsi_\vTheta) : v_\vTheta = B_v \sigma( A_v x + c_v ), \vpsi_\vTheta = B_{\vpsi} \sigma( A_{\vpsi} x + c_{\vpsi} ) \Big\},
\end{equation*}
with $A_v,A_{\vpsi} \in \R^{n \times d}$, $c_v,c_{\vpsi} \in \R^{n \times 1}$, $B_v \in \R^{1\times n}$, $B_{\vpsi} \in \R^{d \times n }$, and $\sigma \colon \R^n \to \R^n$, and $\sigma$ a smooth and bounded non-constant activation function, applied elementwise. We collect all the parameters in $\vTheta \in \R^m$ with $m = 3n(d+1)$.
We remark that, whenever we state that $m \to \infty$, we mean that the number of neurons $n$ is growing to infinity.

Assuming that we are able to construct smooth auxiliary functions $\dD$, $\dN$ and $\vn$ as in Section \ref{subsec:strong_b}, we define the set of discrete admissible functions
\begin{equation} \label{eq:admissible-class}
	\calA_m := \Big\{ \vq_{\vTheta}= (u_{\vTheta},\vphi_{\vTheta}) : u_{\vTheta} = \dD v_{\vTheta} \mbox{ and }		\vphi_{\vTheta} = \vpsi_{\vTheta} - \Big( \frac{  \vpsi_{\vTheta} \cdot \vn }{1 + \dN} \Big) \vn, \mbox{ }(v_{\vTheta},\vpsi_{\vTheta})\in \calC_m \Big\}.
\end{equation}
We remark that the fulfillment of the boundary conditions is guaranteed within the set $\calA_m$, in the sense that $u_\vTheta = 0$ if $\dD = 0$ and $\vphi_\vTheta \cdot \vn = 0$ if $\dN =0$.

\begin{remark}
Naturally, when using Montecarlo integration, one is not allowed to take pointwise evaluations of an arbitrary function $f \in L^2(\omega)$. By density, for every $\epsilon > 0$ we can find a continuous function $f_\epsilon$ with $\| f - f_\epsilon \|_{L^2(\Omega)} < \epsilon$. By the ellipticity of the functional $\calL$ in the $H^1(\Omega) \times H(\Div; \Omega)$ norm (cf. \eqref{eq:funcional_eliptico} below), if we let $\calL_\epsilon$ be the functional \eqref{eq:LS-loss} using $f_\epsilon$ instead of $f$ and $\vq_{0,\epsilon}$ its minimizer, we then have $\| \vq_0 - \vq_{0,\epsilon} \|_{H^1(\Omega) \times H(\Div; \Omega)} < \epsilon$. We can therefore implement the method by using $f_\epsilon$ instead of $f$ and letting $\epsilon \to 0$ as $m \to \infty$.

Nevertheless, we emphasize that, for the sake of the theoretical results in this paper, for any $f \in L^2(\Omega)$ we can take any representative of the equivalence class of $f$ in the definition of the functionals (e.g. in \eqref{eq:reg_func_dis_gen}). Our convergence results are not affected because they are stated in an ``almost sure" sense.     
\end{remark}

\subsection{Approximation properties of neural networks}

Let $\vq_0 = (u_0,\vphi_0) \in \calA$ be the unique minimizer of \eqref{eq:LS-loss}.
We shall make the assumption that $\vq_0$ can be approximated by the neural network spaces. Namely, let us assume that 
\begin{equation} \label{eq:hypothesis}
d(\vq_0,\calA_m) := \inf_{\vq_\vTheta \in \calA_m} \|\vq_0-\vq_\vTheta\|_{H^1(\Omega) \times H(\Div; \Omega)} \to 0 \quad \mbox{as } m\to\infty.
\end{equation}

We briefly comment on this hypothesis. In first place, there are several by now classical results \cite{Cybenko89, hornik1991, Barron93} regarding the approximation properties of neural networks, although without the incorporation of boundary conditions. We additionally point out to \cite{Yarotsky17,He_etal20} for recent results regarding approximation capabilities of ReLU neural networks, including approximation rates.
For deep ReLU neural networks (with at most $\lceil \log_2(d+1) \rceil$ hidden layers), references \cite{He_etal20,Arora18} establish the capability of networks to represent simplicial linear finite element functions, which possess good approximation properties in the $H^1$-norm.
Therefore, if we use a nonconstant activation function $\sigma$, then we expect $d(\vq,\calC_m) \to 0$ when $m\to\infty$ for any $\vq \in H^1(\Omega) \times H(\Div; \Omega)$. 

Condition \eqref{eq:hypothesis} further assumes that the solution $q_0$ can be approximated through the admissible classes $\calA_m$ that incorporate boundary conditions. This hypothesis holds, for example, if one assumes certain regularity of solutions to \eqref{eq:FOS}. For instance, if $u_0 \in C^1(\overline\Omega)$, then it satisfies (recall $\gD = 0$)
\[
\left| \lim_{t \to 0^+} \frac{u_0(z-t\vnu)}{t} \right| = \left| \frac{\pp u_0 }{\pp \vnu} (z) \right| < \infty, \quad z \in \GammaD.
\]
If we write $x = z - t \vnu$, then $t \approx \mbox{dist}(x, \GammaD) \approx \dD(x)$ and the finiteness of the limit above essentially means that $u_0/\dD$ is a bounded function. Additionally,
if we can construct auxiliary functions $\dD$, $\dN$, and $\vn$ in such a way that 
\begin{equation}\label{eq:hipotesis}
	\frac{u_0}{\dD} \in H^1(\Omega), \mbox{ and } \frac{(\vphi_0\cdot\vn)\vn}{\dN} \in H(\Div; \Omega). 
\end{equation} 
then there exists a sequence $\{(v_m,\vpsi_m)\}_{m\in\N}$ with $(v_m,\vpsi_m) \in \calC_m$ for all $m$, such that 

\[
\left\|v_m - \frac{u_0}{\dD} \right\|_{H^1(\Omega)}\to 0 \quad \mbox{and} \quad \left\|\vpsi_m - \sum^{d-1}_{i=1}(\vphi_0\cdot\vt)\vt - \frac{1+\dN}{\dN}(\vphi_0\cdot\vn)\vn\right\|_{H(\Div; \Omega)}\to 0
\]
as $m\to\infty$. Defining the sequence $\{(u_m,\vphi_m)\}_{m\in\N}$ as $u_m = \dD v_m$ and $\vphi_m = \vpsi_m - \Big( \frac{  \vpsi_m \cdot \vn }{1 + \dN} \Big) \vn$, we would have $(u_m,\vphi_m) \in \calA_m$ for all $m$, and $(u_m,\vphi_m) \to (u_0,\vphi_0)$ in $\| \cdot \|_{H^1(\Omega)\times H(\Div;\Omega)}$ and therefore \eqref{eq:hypothesis} would hold. Clearly, \eqref{eq:hipotesis} is a regularity assumption on the solution of \eqref{eq:FOS}, and in turn it translates into its approximability by neural networks.

\subsection{$\Gamma$-convergence}
We aim to prove the convergence of the neural network approximations computed by our method towards minimizers of the least-squares functional $\calL$ in \eqref{eq:LS-loss}. For this purpose, we shall make use of $\Gamma$-convergence theory, that provides a framework for the convergence of functionals. In particular, if one has proven the $\Gamma$-convergence of a sequence of functionals and has a converging sequence of minimizers, then one can guarantee the existence of solutions to the limit problem, as well as the convergence of either minimum values and minimizers. 
We next briefly review the definition and some basic results pertaining to $\Gamma$-convergence and refer to \cite{braides2006} for further details.

\begin{definition}[sequential $\Gamma$-convergence] \label{def:gamma-convergence}
Let $X$ be a metric space and let $F_{n}$, $F: X \to \overline{\R}$, where $\overline{\R}:= [-\infty,+\infty]$. We say that $F_{n}$ $\Gamma$-converges to $F$ (and write $F_n  \xrightarrow[]{\Gamma} F$) if, for every $x \in X$ we have
	\begin{itemize}
		\item \emph{(lim-inf inequality)} for every sequence $\{x_n\}_{n \in \N} \subset X$ converging to $x$,
		\[
		F(x) \le \liminf_{n\to\infty} F_n(x_n) ; 
		\]
		
		\item \emph{(lim-sup inequality)} there exists a sequence $\{x_n\}_{n \in \N}$ converging to $x$ such that
		\[ F(x) \ge \limsup_{n\to\infty} F_n(x_n) .\]
	\end{itemize}
\end{definition}

\begin{definition}[equi-coercivity]\label{def:equicoercividad}
	Let $\{F_n\}_{n \in \N}$ be a sequence of functions $F_n: X \to \overline{\R}$. We say that $\{F_n\}$ is equi-coercive if for all $t \in \R$ there
	exists a compact set $K_t \subset X$ such that $\{F_n \le t\} \subset K_t$.
\end{definition}

\begin{theorem}[fundamental theorem of $\Gamma$-convergence]\label{teo:fund_teo_gamma_conv}
	Let $(X, d)$ be a metric space,	$\{F_n\}_{n\in\N}$ be an equi-coercive sequence of functions on $X$, and $F$ be such that $F_n  \xrightarrow[]{\Gamma} F$. Then,
	$$\exists \min_{X} F = \lim_{n \to\infty} \inf_{X} F_n.$$
	Moreover, if $\{x_n\}_{n\in\N}$ is a precompact sequence in $X$ such that $\lim_{n\to\infty} F_n(x_n) = \lim_{n\to\infty} \inf_{X} F_n$, then every
	limit of a subsequence of $\{x_n\}$ is a minimum point for $F$. 
	
\end{theorem}
We emphasize that the result above guarantees that the equi-coercivity of a family of functionals combined with their $\Gamma$-convergence yields the convergence of the minimizers towards the minimizers of the $\Gamma$-limit.

\subsection{Convergence of the method}
We split the proof of convergence of our method into several steps.
We start by proving the following auxiliary lemma, that shows the continuity of the neural network functions with respect to the parameters.

\begin{lemma}[continuity with respect to neural network parameters] \label{lem:continuidad_tita}
	The map 
	\[\vTheta \mapsto \vq_\vTheta = (u_{\vTheta},\vphi_{ \vTheta}) \in (\calA_m,\| \cdot \|_{H^1(\Omega) \times H(\Div; \Omega)})\] is continuous. Moreover, defining  the functions $G_1, G_2: \R^m \times \Omega \to \R$,
	\begin{equation}\label{def:func_aux}
		G_1(\vTheta,x):=| \vphi_{\vTheta}(x) - \vA \nabla u_{\vTheta}(x) |^2, \quad G_2(\vTheta,x):=|\Div\vphi_{\vTheta}(x) - B u_{\vTheta}(x) + f(x) |^2,
	\end{equation}
for any $R>0$ we have $G_1 \in L^{\infty}(B(0,R) \times \Omega)$ and, assuming $f \in L^{2}(\Omega)$, there exists a function $s \in L^1(B(0,R) \times \Omega))$, depending on $R$, such that $|G_2(\vTheta , x)| \leq s(\vTheta , x)$ for all $(\vTheta , x) \in B(0,R) \times \Omega$.

\end{lemma} 
\begin{proof}
Let us first focus on a generic neural network $v_{\vTheta}: \R^d \to \R$ with one hidden layer,
	\[v_{\vTheta}(x) = B \sigma( A x + c ).\]
Above, we assume $\sigma$ is a Lipschitz continuous activation function, and the parameters $B \in \R^{1\times n}$, $A \in \R^{n\times d}$ and $c \in \R^{n \times 1}$ are collected in $\vTheta \in \R^m$, $m = n (d+2)$. Using the fact that $v_{\vTheta}$ and its derivatives depend continuously  on the parameters, one can verify easily that the map $\R^m \mapsto W^{1,\infty}(\Omega)$ such that $\vTheta \mapsto v_{\vTheta}$ is continuous. Moreover, the function $G: \R^m \times \Omega \to \R$, defined as $G(\vTheta,x) := v_{\vTheta}(x)$ is Lipschitz continuous, and therefore it is bounded on $B(0,R) \times \Omega$ and its (weak) derivatives are essentially bounded on the same set as well. Furthermore, if $f \in L^2(\Omega)$ then $|G(\vTheta,x) + f(x)|^2 \leq 2|G(\vTheta,x)|^2 + 2|f(x)|^2 \leq 2M + 2|f(x)|^2 =: s(\vTheta,x)$, with $s \in L^1(B(0,R) \times \Omega)$.

For arbitrary neural network functions $(u_{\vTheta},\vphi_{ \vTheta})$ in the space $\calA_m$, defined by \eqref{eq:admissible-class}, we exploit the idea above together with the fact that the auxiliary functions $\dD, \dN$ and $\vn$ are smooth to conclude the desired result.	               
\end{proof}

The following lemma guarantees that, for the loss function $\mathcal{L}$ defined in \eqref{eq:LS-loss}, quasi-minimizers over $\calA_m$ converge towards the minimizer $\vq_0 \in \calA$ as $m \to \infty$.

\begin{lemma}[approximation properties of $\calA_m$]\label{lemma:aprox-inf}
For every $m \in \N$, let us define the set of neural network quasi-minimizers 
\[
\calI_m := \{ \vq \in \calA_m : \calL(\vq) \leq \calL(\vq^*) + 1/m \ \forall \vq^* \in \calA_m \}. 
\]
Then, if $\vq_0$ is the unique minimizer of $\calL$ in $\calA$, we have 
\[
\sup_{\vq_m \in \calI_m}\|\vq_m - \vq_0\|_{H^1(\Omega) \times H(\Div; \Omega)} \to 0
\quad  \mbox{as } m \to \infty.
\]     
\end{lemma}
\begin{proof}
	From \cite{Cai94}, we know that $\calL$ is elliptic with respect to the $H^1(\Omega) \times H(\Div; \Omega)$ norm. Namely, there exist positive constants $\alpha$ and $\beta$ such that  
	\begin{equation}\label{eq:funcional_eliptico}
		\alpha\|(u,\vphi)\|_{H^1(\Omega) \times H(\Div; \Omega)} \le \|  \vphi - \vA \nabla u \|_0^2 + \|  \Div(\vphi) - B u \|_0^2 \le \beta\|(u,\vphi)\|_{H^1(\Omega) \times H(\Div; \Omega)}, 
	\end{equation}
for all $(u,\vphi) \in H^1(\Omega) \times H(\Div; \Omega)$.

Let $\eps > 0$. By \eqref{eq:hypothesis}, we consider $m_0 > 0$ such that $d(\vq_0,\calA_m)<\eps$ and $1/m < \eps$ for all $m>m_0$. For every $m > 0$, there exists $\vq^*_{m} = (u^*_{m}, \vphi^*_{m}) \in \calA_{m}$ with $d(\vq_0, \calA_{m}) \ge \|\vq^*_m-\vq_0\|_{H^1(\Omega) \times H(\Div; \Omega)} - \eps$. Then, for all $m>m_0$ and every neural network quasi-minimizer $\vq_m = (u_m, \vphi_m) \in \calI_m$, using that the solution $\vq_0 = (u_0,\vphi_0)$ of \eqref{eq:FOS} satisfies the conditions $\vphi_0 = \vA \nabla u_0$ and $-\Div(\vphi_0) + B u_0 = f$ a.e. in $\Omega$ and exploiting the upper bound in \eqref{eq:funcional_eliptico}, we have
\[ \begin{aligned}
0 & \le \calL(\vq_m) \le \calL(\vq^*_m) + \eps = \|  \vphi^*_m - \vA \nabla u^*_m \|_0^2 + \|  \Div(\vphi^*_m) - B u^*_m +f\|_0^2 + \eps \\
& =  \|  \vphi^*_m - \vphi_0 - \vA \nabla (u^*_m - u_0) \|_0^2 + \|  \Div(\vphi^*_m - \vphi_0) - B (u^*_m - u_0) \|_0^2 + \eps\\ 
& \le \beta  \|\vq^*_m-\vq_0\|_{H^1(\Omega) \times H(\Div; \Omega)} + \eps \le \beta ( d(\vq_0, \calA_m) + \eps) + \eps = (2\beta+1)\eps.
\end{aligned} \] 

Finally, by combining this estimate with the lower bound in \eqref{eq:funcional_eliptico}, and exploiting the fact that $\vq_0$ satisfies \eqref{eq:FOS} a.e. in $\Omega$, we reach the estimate 
\[ \begin{aligned}
 \|\vq_m - \vq_0\|_{H^1(\Omega) \times H(\Div; \Omega)} & \le \frac1\alpha \Big(  \|  \vphi_m - \vphi_0 - \vA \nabla (u_m - u_0) \|_0^2 + \|  \Div(\vphi_m - \vphi_0) + B (u_m - u_0) \|_0^2  \Big) \\
& \le  \frac1\alpha \Big( \|  \vphi_m - \vA \nabla u_m \|_0^2 + \|  \Div(\vphi_m) + B u_m +f\|_0^2 \Big) = 
\frac{\calL(\vq_m)}{\alpha} < \frac{(2\beta + 1)}{\alpha}\eps,
 \end{aligned} \]
for every $\vq_m \in \calI_m$ and $m > m_0$. Since $\eps$ is arbitrary small, this concludes the proof.
\end{proof}

The result above assumes that, given $\vTheta \in \R^m$, one can compute $\calL(u_{\vTheta},\vphi_{ \vTheta})$ exactly. This is not the case in general, because we resort to Monte Carlo integration for the computation of the $L^2$ norms in \eqref{eq:LS-loss}; cf. the discrete loss functional \eqref{eq:loss_discrete}. To deal with this issue, we consider a regularized version of the loss functions $\calL$ and $\calL_N: \calA_m \to \overline\R$, using $\R^m$ as domain. Given $R>0$, we define the regularized functional $L: \R^m \to \R$ as
\begin{equation}\label{eq:reg_func}
	L(\vTheta) := 
	\left\lbrace
	\begin{array}{ll}
		\calL( u_{\vTheta}, \vphi_{\vTheta} ) & \mbox{if } |\vTheta| \le R, \\
		+\infty & \mbox{otherwise. } \\
	\end{array}
	\right.
\end{equation}

Next, we let $\{X_i\}_{i \in \mathbb{N}}$ be an i.i.d. sequence of random variables, defined on a probability space $(\Lambda,\Sigma,P)$ with $X_i:\Lambda \to \Omega \quad \forall i \in \mathbb{N}$, with uniform  probability density on $\Omega$. Given $\lambda \in \Lambda$, $R>0$, and $N \in \N$ we set $V_N(\lambda) := \cup_{i \le N} \{X_i(\lambda)\} $, and the regularized discrete functional $L_{\lambda,N}: \R^m \to \overline\R$ as
\begin{equation}\label{eq:reg_func_dis}
	L_{\lambda,N}(\vTheta) := 
	\left\lbrace
	\begin{aligned}
		& \frac{|\Omega|}{N}\sum_{ x \in V_N(\lambda) } G_1(\vTheta,x) + G_2(\vTheta,x) & \mbox{ if } |\vTheta| \le R, \\
		& +\infty & \mbox{otherwise, } \\
	\end{aligned}
	\right.
\end{equation}
with $G_1$ and $G_2$ as in \eqref{def:func_aux}. 

With these definitions, we can prove the 
pointwise $P$-almost sure convergence of the sequence $\{L_{\lambda,N}\}_{N \in \N}$ towards $L$.

\begin{lemma}[almost sure convergence of regularized discrete loss functions]\label{lemma:conv_puntual}
	Consider $R>0$, $L$ as in \eqref{eq:reg_func}, $L_{\lambda,N}$ and $\{X_i\}_{i \in \mathbb{N}}$ an i.i.d. family of random variables defined in the probability space $(\Lambda,\Sigma,P)$ as in \eqref{eq:reg_func_dis}. Then $L_{\lambda,N}(\vTheta) \to L(\vTheta)$ as $N \to \infty$ $P$-almost surely, for all $\vTheta \in \R^m$. 
\end{lemma}
\begin{proof}
Since we are using the same parameter $R$ in the definitions of $L$ and $L_{\lambda,N}$, if $|\vTheta| > R$ we have $L(\vTheta) = L_{\lambda,N}(\vTheta) = +\infty$ and there is nothing to be proven. We therefore assume $|\vTheta| \le R$. Recalling $V_N(\lambda) = \cup_{i \le N} \{X_i(\lambda)\} $ with $\lambda \in \Lambda$ and the definition of $G_1$ and $G_2$ \eqref{def:func_aux}, an application of the strong law of large numbers yields
\[
\frac{|\Omega|}{N}\sum_{ x \in V_N(\lambda) } | \vphi(x) - \vA \nabla u(x) |^2 \xrightarrow[N \to \infty]{a.s.} \int_{\Omega} | \vphi - \vA \nabla u |^2,	
\]
and     	
\[
\frac{|\Omega|}{N}\sum_{ x \in V_N(\lambda) } | \Div\vphi(x) - B u(x) + f(x) |^2 \xrightarrow[N \to \infty]{a.s.} \int_{\Omega} | \Div\vphi - B u + f |^2	
\]
for all $(u,\vphi) \in \calA_m$. It follows immdiately that $L_{\lambda,N}(\vTheta) \to L(\vTheta)$ $P$-almost surely as $N \to \infty$.
\end{proof}

We are now in position to prove the almost sure $\Gamma$-convergence of $L_{\lambda,N}$ to $L$ as the number of quadrature points $N \to \infty$.

\begin{theorem}[almost sure $\Gamma$-convergence]\label{teo:gamma_conv}
Let $R>0$, $L$ be as in \eqref{eq:reg_func}, and $L_{\lambda,N}$ and $\{X_i\}_{i \in \mathbb{N}}$ be an i.i.d. family of random variables defined in the probability space $(\Lambda,\Sigma,P)$ as in \eqref{eq:reg_func_dis}. Then, assuming $f \in L^2(\Omega)$, it holds that $L_{\lambda,N} \xrightarrow[]{\Gamma} L$ as $N \to \infty$ $P$-almost surely.
\end{theorem}
\begin{proof}
We first observe that the lim-sup inequality is a trivial corollary of Lemma \ref{lemma:conv_puntual}. Indeed, it suffices to consider the recovery sequence $\{\vTheta_N\}_{N \in \N} \subset \R^m$, $\vTheta_N \equiv \vTheta$, and by Lemma \ref{lemma:conv_puntual} we have $L_{\lambda,N}(\vTheta_N) \to L(\vTheta)$ with $N \to \infty$ $P$-almost surely.   

We next prove the lim-inf inequality. Given $\vTheta \in \R^m$, let $\{\vTheta_N\}_{N \in \N} \subset \R^m$ be a sequence of parameters such that $\vTheta_{N} \to \vTheta$. We aim to prove that  
	\begin{equation}\label{eq:liminf}
		L(\vTheta) \le  \liminf_{N \to \infty} L_{\lambda,N}(\vTheta_N).
	\end{equation} 
We observe that, if $|\vTheta|>R$ then there exists $N_0 = N_0(\lambda)$ such that $L(\vTheta)=L_{\lambda,N}(\vTheta_N) = +\infty$ for all $N>N_0$, and \eqref{eq:liminf} trivially holds. Therefore, without loss of generality we assume $\{\vTheta_N\}_{N \in \N} \subset \overline{B(0,R)}$. In that case, we extract a subsequence in such a way that $L_{\lambda,N}(\vTheta_N) \to \liminf_{N \to \infty} L_{\lambda,N}(\vTheta_N)$ and, for the sake of simplicity, we omit the relabeling. By Lemma \ref{lem:continuidad_tita}, the map $\vTheta \mapsto (u_{\vTheta},\vphi_{ \vTheta}) \in (\calA_m,\| \cdot \|_{H^1(\Omega) \times H(\Div; \Omega)})$ is continuous and therefore $(u_{\vTheta_N}, \vphi_{\vTheta_N}) \to (u_{\vTheta}, \vphi_{\vTheta})$ in the $H^1(\Omega) \times H(\Div; \Omega)$ norm. Because $\Omega$ is bounded, this implies 
\[ \begin{aligned}
& \|u_{\vTheta_N}-u_{\vTheta}\|_{L^1(\Omega)} \to 0, & \quad \|\nabla u_{\vTheta_N}-\nabla u_{\vTheta}\|_{L^1(\Omega)} \to 0,\\  
& \|\vphi_{\vTheta_N}-\vphi_{\vTheta}\|_{L^1(\Omega)} \to 0, \quad \mbox{and} & \ \|\Div \vphi_{\vTheta_N}-\Div \vphi_{\vTheta}\|_{L^1(\Omega)} \to 0.
\end{aligned}\]
Then, defining $G_1$ and $G_2$ as in \eqref{def:func_aux}, we extract another subsequence in such a way that $G_1(\vTheta_N,x) + G_2(\vTheta_N,x) \to G_1(\vTheta,x) + G_2(\vTheta,x)$ almost everywhere in $\Omega$, and, as before, we omit the relabeling. 

In order to prove \eqref{eq:liminf}, we are going to show that the latter subsequence satisfies $L_{\lambda,N}(\vTheta_N) \to L(\vTheta)$ with $N \to \infty$ $P$-almost surely. Let $\eps>0$ be an arbitrary number, using the triangle inequality, we split
\begin{equation} \label{eq:split-L}
|L_{\lambda,N}(\vTheta_N) - L(\vTheta)| \le |L_{\lambda,N}(\vTheta_N) - L_{\lambda,N}(\vTheta)| +  |L_{\lambda,N}(\vTheta) - L(\vTheta)|.
\end{equation}   
From Lemma \ref{lemma:conv_puntual}, it follows that $|L_{\lambda,N}(\vTheta) - L(\vTheta)| \to 0$ $P$-almost surely. Thus, there exists $N_0=N_0(\lambda)$ such that $|L_{\lambda,N}(\vTheta) - L(\vTheta)|\le \eps/4$ for all $N>N_0$. 

In order to bound the first term in the right hand side in \eqref{eq:split-L}, we first observe that Lemma \ref{lem:continuidad_tita} shows that $G_1$ is uniformly bounded and $G_2$ is bounded above by some integrable function. Thus, there exists $s \in L^1(\Omega)$, depending on $R$, such that 
\begin{equation}\label{eq:gammac1}
\big| G_1(\vTheta_N,x) + G_2(\vTheta_N,x) - G_1(\vTheta,x) - G_2(\vTheta,x) \big| \leq s(x),  
\end{equation}
for all $(\vTheta,x) \in B(0,R) \times \Omega$.
Now we apply Egorov's Theorem to construct a set $\calK \subset \Omega$ such that $\int_{\calK} s(x) dx < \eps/8$ and $G_1(\vTheta_N, \cdot) + G_2(\vTheta_N,\cdot) \to G_1(\vTheta,\cdot) + G_2(\vTheta,\cdot)$ uniformly in $\Omega \setminus \calK$. We bound
\[
|L_{\lambda,N}(\vTheta_N) - L_{\lambda,N}(\vTheta)| \le A_1 + A_2,
\]
where
\[ \begin{aligned}
A_1 & = \frac{|\Omega|}{N}\sum_{ x \in V_N(\lambda) \cap (\Omega\setminus \calK) } \big| G_1(\vTheta_N,x) + G_2(\vTheta_N,x) - G_1(\vTheta,x) - G_2(\vTheta,x) \big|, \\
A_2 & = \frac{|\Omega|}{N}\sum_{ x \in V_N(\lambda) \cap \calK } \big| G_1(\vTheta_N,x) + G_2(\vTheta_N,x) - G_1(\vTheta,x) -  G_2(\vTheta,x) \big|.
\end{aligned} \]
Using the uniform convergence in $\Omega \setminus \calK$, $P$-almost surely there exists $N_1 = N_1(\lambda)$ such that, if $N > N_1$, then $\big| G_1(\vTheta_N,x) + G_2(\vTheta_N,x) - G_1(\vTheta,x) -  G_2(\vTheta,x) \big| < \frac{\eps}{4|\Omega|}$ for all $x \in \Omega \setminus \calK$. Then, it follows that $A_1 < \eps/4$ if $N>N_1$. 

On the other hand, we use \eqref{eq:gammac1} to derive
\[
A_2 
\le  \frac{|\Omega|}{N}\sum_{ x \in V_N(\lambda) } \chi_\calK(x) s(x).
\]
By the strong law of large numbers, we have 
\[
\frac{|\Omega|}{N}\sum_{ x \in V_N(\lambda) } \chi_\calK(x) s(x) \xrightarrow[N \to \infty]{a.s.} \int_{\calK} s(x) < \frac{\eps}{8}.
\]
Therefore, $P$-almost surely there exists $N_2 = N_2(\lambda)$ such that, if $N>N_2$ then 
\[\Big| \frac{|\Omega|}{N}\sum_{ x \in V_N(\lambda) } \chi_\calK(x) - \int_{\calK} s(x) \Big| < \frac{\eps}{8},
\] 
which implies that $\frac{|\Omega|}{N}\sum_{ x \in V_N(\lambda) } \chi_\calK(x) s(x) < \frac{\eps}{4}$.
Consequently, we have $A_2 <  \frac{\eps}{4}.$

Collecting the estimates above, it follows that $P$-almost surely we can choose $N' = N'(\lambda) = \max \{N_0,N_1,N_2\}$ such that 
$$|L_{\lambda,N}(\vTheta_N) - L(\vTheta)| \le |L_{\lambda,N}(\vTheta_N) - L_{\lambda,N}(\vTheta)| +  |L_{\lambda,N}(\vTheta) - L(\vTheta)| \leq \eps,$$
for all $N > N'$. This shows that \eqref{eq:liminf} holds, and concludes the proof.
\end{proof}

The following theorem is the main result of this section and it roughly states that, if we have a reasonable procedure for the minimization of $L_{\lambda,N}$ on $\calA_m$, then we can expect convergence to the solution $\vq_0$.

\begin{theorem}[convergence] \label{teo:conv}
Suppose that for any fixed $m \in \N$ and $R>0$ we can construct a sequence $\{\vTheta_N\}_{N \in \N} \subset B(0,R) \subset \R^m$ such that  $\lim_{N \to \infty} L_{\lambda,N}(\vTheta_N) = \lim_{N \to \infty} \inf_{\vTheta \in \R^m} L_{\lambda,N}(\vTheta)$, with $L_{\lambda,N}$ defined as in \eqref{eq:reg_func_dis}. Let $(u_0,\vphi_0) = \vq_0 = \arg \min_{\vq \in \calA} \calL(\vq)$. Given $\eps>0$, there $P$-almost surely exist $m_0=m_0(\eps) \in \N$, $R=R(m_0)>0$ and $N_0 = N_0(m_0) \in \N$ such that, if one constructs a sequence $\{\vTheta_N\}_{N \in \N}$ as above, then  
\[\|(u_0,\vphi_0) - (u_{\vTheta_N} , \vphi_{\vTheta_N})\|_{H^1(\Omega) \times H(\Div; \Omega)} \le \eps \quad \mbox{for all } N>N_0,
\]
where $(u_{\vTheta_N} , \vphi_{\vTheta_N})$ is the neural network function defined by the parameters $\vTheta_N$.
\end{theorem}

\begin{proof}
Let $\eps > 0$ and consider the set of neural network quasi-minimizers introduced in Lemma \ref{lemma:aprox-inf}, $\calI_m = \{ \vq \in \calA_m : \calL(\vq) \leq \calL(\vq^*) + 1/m \ \forall \vq^* \in \calA_m \}$. By that lemma, there exists $m_0 > 0$ such that 
\begin{equation}\label{eq:teo_conv_1}
	\| \vq_0 - \vq_{m_0}\|_{H^1(\Omega) \times H(\Div; \Omega)} < \eps/2,
\end{equation}
for all $\vq_{m_0} \in \calI_{m_0}$. Next, we fix $R_0>0$ large enough so that there exists $\vTheta \in B(0,R_0)$ with $\vq_\vTheta^* = (u_{\vTheta}^*,\vphi_{ \vTheta}^*) \in \calI_{m_0}$. For the functional $L$ defined in \eqref{eq:reg_func}, this implies that $q_{\vTheta} \in \calI_m$ for all $\vTheta \in \arg \min_{\vTheta \in B(0,R_0)} L(\vTheta)$.   

For this choice of $m_0$ and $R_0$, from Theorem \ref{teo:gamma_conv} we have $L_{\lambda,N} \xrightarrow[]{\Gamma} L$ P-almost surely. From the definition of $L_{\lambda,N}$  \eqref{eq:reg_func_dis}, it follows immediately that $\{L_{\lambda , N}\}_{N \in \N}$ is an equi-coercive sequence, according to Definition \ref{def:equicoercividad}. Therefore, 
we deduce that $P$-almost surely there exists $N_0 > 0$ such that
\begin{equation}\label{eq:teo_conv_2}
\| (u_{\vTheta_N},\vphi_{\vTheta_N}) - \vq_{m_0} \|_{H^1(\Omega) \times H(\Div; \Omega)}  < \eps/2
\end{equation}
for all $N>N_0$ for some $\vq_{m_0} \in \calI_{m_0}$. This bound follows by Theorem \ref{teo:fund_teo_gamma_conv} because every cluster point of $\{ \vTheta_N\}$ is a minimum point for $L$, and because of the continuity of the map $\vTheta \mapsto (u_{\vTheta},\vphi_{ \vTheta})$.

The proof concludes upon combining \eqref{eq:teo_conv_1} and \eqref{eq:teo_conv_2}.    
\end{proof}

\section{General framework} \label{sec:general}
In this section, we extend the theoretical analysis we performed in Section \ref{sec:analysis} and put it into an abstract framework. Afterwards, we illustrate how such a framework applies to some well-established unstructured neural-network methods for the approximation of PDEs.
 
Let $\Omega \subset \R^d$ and $\gamma \in \N$. We assume our problem is posed in some admissible vector space
\[
\calA \subset W^{\gamma,1}_{loc}(\Omega; \R^n),
\]
namely, that every function $\vq \in \calA$ has locally integrable weak derivatives of order up to $\gamma$. The space $\calA$ may or may not include boundary conditions or constraints of any type. In the setting we described in Section \ref{sec:introduction}, the target dimension is $n = 1+d$, the differentiability index is $\gamma=1$, and we identify $\calA \ni \vq = (u, \vphi)$.
Additionally, we assume the space $\calA$ is furnished with some norm $\| \cdot \|_{\calA}$, which in our setting corresponds to the $H^1(\Omega) \times H(\Div; \Omega)$-norm.

We consider $\omega_1,...,\omega_K$ Borel subsets of $\overline\Omega$, each $\omega_i$ furnished with a finite Radon measure $\mu_i$, and some given functions $f_1,...,f_{n_f}$ with $f_i:\Omega \to \R$. 
 Given some integrable functions $F_i: \R^{n_\gamma n +n_f+d} \to \R$, $1 \le i \le K$, we define the loss functional 
\begin{equation*}
	\calL(\vq) := \sum^K_{i=1} \int_{\omega_i} F_i(D^{\alpha_1}\vq, \ldots ,D^{\alpha_{n_\gamma}}\vq, f_1, \ldots ,f_{n_f},x) \, d\mu_i,
\end{equation*}
with $F_i$ in such a way that all the integrals involved are well defined. Namely, we assume the loss functional consists of $K$ terms, each of which may be defined on different subdomains of $\overline\Omega$. Each of these terms involves certain partial derivatives of $\vq$ of order up to $\gamma$.
The subdomains $\omega_i$ need not be open; for example, we could allow for $\omega_i \subset \partial\Omega$ and the corresponding term would be able to accommodate boundary data. In such a case, the corresponding trace operator must be bounded on the space $\calA$.

Consider now a space $\calA_m \subset \calA$ in such a way that we have a surjective map $\vTheta: \R^m \mapsto \calA_m$. In the setting from Section \ref{sec:analysis}, this space consists of the functions obtained through a neural network with a modification to account for boundary conditions, cf. \eqref{eq:admissible-class}. We denote by $\vq_{\vTheta}$ a generic element of $\calA_m$. For $1 \le i \le K$, we define $G_i(\vTheta,x): \R^n \times \Omega \to \R$ as
\begin{equation*}
	G_i(\vTheta,x) = F_i(D^{\alpha_1}\vq_{\vTheta},\ldots,D^{\alpha_{n_\gamma}}\vq_{\vTheta},f_1,\ldots, f_{n_f},x)
\end{equation*}
and, given $R>0$, we define the regularized loss functional $L: \R^m \to \R$
\begin{equation}\label{eq:reg_func_gen}
	L(\vTheta) := 
	\left\lbrace
	\begin{array}{ll}
		\calL( \vq_{\vTheta} ) & \mbox{if } |\vTheta| \le R, \\
		+\infty & \mbox{otherwise}. \\
	\end{array}
	\right.
\end{equation} 

Let $\{X^1_j\}_{j \in \mathbb{N}},...,\{X^K_j\}_{j \in \mathbb{N}}$ i.i.d. sequences of random variables, defined in the probability space $(\Lambda,\Sigma,P)$ with $X^i_j:\Lambda \to \omega_i \quad \forall j \in \mathbb{N}$, $1 \le i \le K$, in such a way that the probability density $\overline\mu^i$ of $X^i_j$ is distributed as $\mu_i$ on $\omega_i$, that is 
\[
\overline\mu^i(E) = \frac{\mu_i(E)}{\mu_i(\omega_i)} \quad \mbox{for every Borel set } E \subset \omega_i .
\]

Given $\lambda \in \Lambda$, $R>0$, and $N \in \N$ we define the sampling nodes $V^i_N(\lambda) := \cup_{j \le N} \{X^i_j(\lambda)\} $, and the regularized discrete loss functional $L_{\lambda,N}: \R^m \to \R$,
\begin{equation}\label{eq:reg_func_dis_gen}
	L_{\lambda,N}(\vTheta) := 
	\left\lbrace
	\begin{aligned}
		& \sum^K_{i=1}\frac{\mu_i(\omega_i)}{N} \sum_{ x \in V^i_N(\lambda) } G_i(\vTheta,x) & \mbox{if } |\vTheta| \le R, \\
		& +\infty & \mbox{otherwise.} \\
	\end{aligned}
	\right.
\end{equation}

In order to extend our convergence estimates in Section \ref{sec:analysis} to a general framework,
we consider the following hypotheses:
\begin{itemize}
	\item[(H1)] The map $\R^m \mapsto (\calA_m,\|\cdot\|_{\calA})$ with $\vTheta \mapsto \vq_\vTheta$ is  continuous.

	\item [(H2)] 
For all $1 \le i \le K$ and every convergent sequence $\{\vq_{\vTheta_n}\}_{n\in\N} \subset \calA_m$, with $\vq_{\vTheta_n} \to \vq_{\vTheta} \in \calA_m$ with respect to the $\calA$-norm, there exists a subsequence $\{\vq_{\vTheta_{n_j}}\}_{j\in\N}$ such that $G_i(\vTheta_{n_j},x) \to G_i(\vTheta,x)$ $\mu_i$-almost everywhere.

	\item [(H3)] For every $R>0$, there exist functions $s_i \in L^1_{\mu_i}(\omega_i)$ such that $|G_i(\vTheta,x)| \le s_i(x)$ for all $1 \le i \le K$, for all $\vTheta \in B(0,R)$, and $\mu_i$-almost every $x \in \omega_i$.

	\item [(H4)] The loss function $\calL$ has a unique minimizer $\vq_0 \in \calA$.
		
	\item [(H5)] Let $\calI_m := \{ \vq \in \calA_m : \calL(\vq) \leq \calL(\vq^*) + 1/m \ \forall \vq^* \in \calA_m \}$ be the set of neural network quasi-minimizers. Then, $\sup_{\vq_m \in \calI_m} \|\vq_m - \vq_0\|_{\calA} \to 0$ as $m \to \infty$.

\end{itemize}

Let us comment on these assumptions and how they relate to our analysis in the previous section. Hypothesis (H1) corresponds to the first part in the conclusion of Lemma \ref{lem:continuidad_tita}, and guarantees the stability of neural network functions with respect to the parameters. Hypothesis (H2) roughly states that, for neural network functions, one can pass from convergence in $\calA$ to almost everywhere convergence (up to a subsequence). In our setting, we showed this condition to hold in the proof of Theorem \ref{teo:gamma_conv}. Our assumption (H3) requires the existence of an $L^1$-upper bound for the terms $G_i$. This condition appeared in the second part of Lemma \ref{lem:continuidad_tita}. The ellipticity of the loss functional $\calL$ guarantees that hypothesis (H4) is satisfied. Finally, hypothesis (H5) involves the approximability of the solution to the continuous problem by the neural network quasi-minimizers of $\calL$. In our setting, this appeared in Lemma \ref{lemma:aprox-inf}, and is a consequence of ellipticity and assumption \eqref{eq:hypothesis}. Clearly, in the definition of the set $\calI_m$, one can equivalently use any positive sequence $a_m \searrow 0$ instead of our default choice $a_m = 1/m$.

The following two results extend Theorem \ref{teo:gamma_conv} and Theorem \ref{teo:conv}, respectively; we outline the main steps of their proofs. We first address the $\Gamma$-convergence of the regularized discrete functionals.

\begin{theorem}[almost sure $\Gamma$-convergence, general case]\label{teo:gamma_conv_gen}
Let $R>0$, and $L$, $L_{\lambda,N}$ be as in \eqref{eq:reg_func_gen} and \eqref{eq:reg_func_dis_gen}, respectively. Then, under assumptions (H1), (H2), and (H3), it holds that $L_{\lambda,N} \xrightarrow[]{\Gamma} L$ with $N \to \infty$ P-almost surely.
\end{theorem}
\begin{proof}
The arguments used in the proof of Theorem \ref{teo:gamma_conv} can be easily adapted to this case. 
Indeed, the lim-sup inequality follows trivially by taking the recovery sequence $\{\vTheta_N\}_{N \in \N}$, $\vTheta_N \equiv \vTheta$ and using a strong law of large numbers.

To prove the lim-inf inequality, we start from a bounded sequence of parameters $\{ \vTheta_N\}_{N \in \N}$ and use (H1) to extract a converging subsequence $\{ \vq_{\vTheta_N}\}_{N \in \N}$ in the $\calA$-norm. Then, by (H2) we can extract another subsequence such that $G_i(\vTheta_{n_j},x) \to G_i(\vTheta,x)$ $\mu_i$-almost everywhere for all $1 \le i \le K$ and by (H3) we know that every function $G_i$ has an upper bound in $L_{\mu_i}^1( \Omega)$. The conclusion then follows by applying Egorov's Theorem on every subset $\omega_1, \ldots \omega_K$.
\end{proof}

Once we have the almost sure $\Gamma$-convergence of the regularized discrete functionals, the convergence of the neural network minimizers can be proved by arguing as in Theorem \ref{teo:conv}.

\begin{theorem}[convergence, general case] \label{teo:conv_general} 
Assume hypotheses (H1)--(H5) are satisfied, and
suppose that for any fixed $m \in \N$ and $R>0$ we can construct a sequence $\{\vTheta_N\}_{N \in \N} \subset B(0,R) \subset \R^m$ such that  $\lim_{N \to \infty} L_{\lambda,N}(\vTheta_N) = \lim_{N \to \infty} \inf_{\vTheta \in \R^m} L_{\lambda,N}(\vTheta)$, with $L_{\lambda,N}$ defined as in \eqref{eq:reg_func_dis}.
Let $\vq_0 = \arg \min_{\vq \in \calA} \calL(\vq)$. Given $\eps>0$ there exist $m_0=m_0(\eps) \in \N$, $R=R(m_0)>0$ and $N_0 = N_0(m_0) \in \N$ P-almost surely, such that 
\[
\| \vq_0 - \vq_{\vTheta_N}\|_{\calA} \le \eps \quad \mbox{for all } N>N_0,
\]
where $\vq_{\vTheta_N} \in \calA_{m_0}$ is the neural network function defined by the parameters $\vTheta_N$.   
\end{theorem}

\begin{proof}
We first remark that hypothesis (H4) is needed to guarantee the existence of a well-defined minimizer $\vq_0 \in \calA$, and therefore (H5) is meaningful.
Given $\eps > 0$, we use hypothesis (H5) to find $m_0 \in \N$ such that, if $\vq_{m_0} \in \calI_{m_0}$ then  $ \| \vq_0 - \vq_{m_0}\|_{\calA} < \eps/2$. 

Next, we fix $R_0>0$ large enough so that there exists $\vTheta \in B(0,R_0)$ with $\vq_{ \vTheta} \in \calI_{m_0}$, and use this $R_0$ in Theorem \ref{teo:gamma_conv_gen} to deduce that $L_{\lambda,N} \xrightarrow[]{\Gamma} L$ with $N \to \infty$ P-almost surely.
The result then follows by the equi-coercivity of the sequence $\{L_{\lambda , N}\}_{N \in \N}$ by applying the fundamental theorem of $\Gamma$-convergence (Theorem \ref{teo:fund_teo_gamma_conv}).
\end{proof}

We next discuss how two well-known methods fit into the framework in hypotheses (H1)--(H5), and thus Theorem \ref{teo:conv_general} establishes their convergence.

\begin{remark}[Deep Ritz Method] \label{rem:DRM}
The DRM was proposed by E and Yu in \cite{E18}, and is tailored for numerically solving variational problems. A prototypical example is the homogeneous Dirichlet problem, that corresponds to the minimization of the energy $\calL \colon H^1_0(\Omega) \to \R$,
\[
\calL(u) = \frac12 \int_\Omega |\nabla u|^2 - \int_\Omega f u.
\]
We assume $ \| f \|_{L^{2}(\Omega)} <\infty $, consider $\calA = H^1_0(\Omega)$, and define the neural network spaces $\calA_m$ as in \eqref{eq:admissible-class}. Arguing as in Section \ref{sec:analysis}, it is possible to show that hypotheses (H1)--(H4) hold for this loss function. Indeed, (H1) and (H3) can be proved in the same fashion as Lemma \ref{lem:continuidad_tita}, while (H2) follows because for every bounded sequence in $H^1_0(\Omega)$ we can extract an almost everywhere convergent subsequence, and (H4) is a standard PDE result. Finally, hypothesis (H5) can be obtained from classical approximation results \cite{Cybenko89, hornik1991, Barron93, Yarotsky17,He_etal20}.
\end{remark}

\begin{remark}[Deep Galerkin Method] \label{rem:DGM}
The DGM was introduced by Sirignano and Spiliopoulos in \cite{DGM18}, and uses as loss functional the $L^2$-norm of the PDE residual on the neural network functions. Within the
convergence framework in \cite[Section 7]{DGM18}, and the conditions assumed there, we set
$\calA: = \mathcal{C}^{0, \delta, \delta/2}(\overline {\Omega}_{T}) \cap L^2((0, T]; W^{1,2}_0 (\Omega)) \cap W^{(1,2), 2}_0 (\Omega_T ') $, where $ \delta> 0 $ and $ \Omega_T'$ is any interior subdomain of $\Omega_T$, cf. Theorem 7.3. We furnish this space with the $\| \cdot \|_{H^2 (\Omega_T)}$ norm, and define $\calA_m$ according to \eqref{eq:admissible-class}.

Then, assumptions (H1) and (H3) can be verified by arguing as in Lemma \ref{lem:continuidad_tita} by requiring suitable regularity assumptions on the initial and boundary data and parameters of the equation; for example, these hold straightforwardly for these data and parameters are bounded.
Hypothesis (H2) can be proved by using the boundedness of $\Omega_T$ and arguing as in the proof of Theorem \ref{teo:gamma_conv} to exploit the convergence properties of the $H^2(\Omega)$-norm. Finally, hypotheses (H4) and (H5) are addressed in \cite[Theorem 7.3]{DGM18}. We observe that despite this result is stated for a single minimizing sequence $\{f^n\}$, defined in \cite[Theorem 7.1]{DGM18}, the arguments applies "uniformly" to any possible construction of $\{f^n\}$, and then (H5) is verified. Finally, we point out that the convergence of discrete minimizers of $\calL$ is proven in the weaker norm $ \| \cdot \|_{L^{\rho}(\Omega_T)}$, with $\rho <2$. Therefore, our conclusion in Theorem \ref{teo:conv_general} is valid if we measure convergence in such a norm.
\end{remark}

\section{Numerical experiments} \label{sec:numerical}

In this section, we present numerical results for the method we proposed in Section \ref{sec:description}. We did not prioritize any particular neural network architecture, and used between one- and five-layer networks with sigmoidal activation functions to construct $u_{\vTheta}$ and $\vphi_{\vTheta}$. For the construction of the auxiliary functions $\vn$, $\dD$, $\dN$, $\GD$, $\GN$, we used between one and three-layer networks with less neurons per layer. In the training process, we used the ADAM \cite{ADAM} algorithm to update the parameters, with a decaying learning rate schedule.
  
We observe an improvement in the method's performance when
explicit approximations of the auxiliary functions $\dD$ and $\dN$ are used. These functions, which depend on the geometry of the domain, are many times explicitly available in practice.

We recall that, as explained in sections \ref{sec:analysis} and \ref{sec:general}, the numerical solution depends on the number of degrees of freedom $m$ and the number of collocation points $N$. Both must go to infinity to guarantee convergence. In all the numerical examples we show below, these quantities remain fixed. Therefore, in these examples the convergence as a function of the iterations occurs towards the minimizer of the discrete loss functional $L_{\lambda,N}$ (cf. \eqref{eq:reg_func_dis}) corresponding to the values of $m$ and $N$ we have set.

\begin{example}[Laplace operator]
We consider the following problem in arbitrary dimension. Let $\Omega = \{ x\in \Rd : \ -1< x_1,..., x_d< 1\}$, $\GammaN = [-1,1]^{d-1} \times \{1\},$ and $k \in \mathbb{N}$. We seek $u \colon \Omega \to \R$ such that
\begin{equation}\label{eq:example1}
\left\lbrace \begin{aligned}
	-\Delta u & = \prod^{d-1}_{i=1}\sin(k \pi x_i) \left( (d-1)k^2 \pi^2 (1-x_d^2) + 2 \right) \quad & \mbox{in } \Omega,\\
	u & = 0 \quad & \mbox{on } \partial\Omega \setminus \GammaN,\\
	\nabla u \cdot \vnu & = -2\prod^{d-1}_{i=1}\sin(k \pi x_i) \quad & \mbox{on } \GammaN.\\
\end{aligned} \right.
\end{equation}
Here, we have $\vnu = (0,...,0,1)$ on $\GammaN$, and the solution to \eqref{eq:example1} is
\[
u = \prod^{d-1}_{i=1} \sin(k \pi x_i)(1-x^2_d).
\]

We point out that the parameter $k$ is a frequency that allows us to choose how oscillatory the exact solution $u$ is.
We first tested the method in a two-dimensional domain ($d=2$). Figure \ref{fig:dim2} displays the results we obtained for $k=1$ and by constructing $u_{\vTheta}$ and $\vphi_{\vTheta}$ using neural networks with $15$ sigmoidal activation functions per layer.
 At the end of the stochastic gradient descent algorithm we computed the value $L_N (\vTheta) = 0.0450$. Taking into account the ellipticity of the loss function $\calL$, arguing as in Lemma \ref{lemma:aprox-inf} we deduce 
 \[
 \calL (\vq_m) \simeq \| \vq_m - \vq_0 \|_{H^1(\Omega) \times H(\Div; \Omega)} ,
 \] 
and therefore this quantity serves as an error estimator.

Figure \ref{fig:dim2_k2} corresponds to $k=2$, and we used a similar architecture, but with $18$ sigmoidal activation functions per layer. We observed a fast convergence in the number of iterations, reaching $L_N (\vTheta) = 1.89$ by the end of the minimization algorithm.
Finally, Figure \ref{fig:altadim} reports the results we obtained in case $d = 5$, $k=1$. In this case, we used networks with $25$ sigmoidal activation functions per layer and obtained $L_N (\vTheta) = 2.32$.

\begin{figure}[h]
	\centering
	\begin{tabular}{|c|c|c|c|}
		\hline
		\subf{\includegraphics[width=50mm]{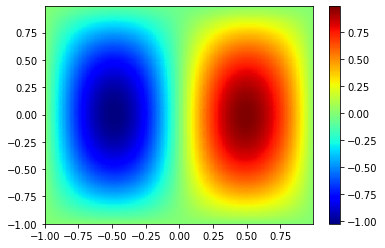}}
		{$ u_{\vTheta}$}
		&
		\subf{\includegraphics[width=50mm]{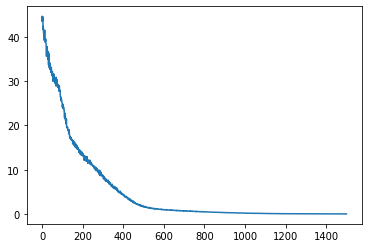}}
        {Loss function vs. iterations.}
		\\
		\hline
		\subf{\includegraphics[width=50mm]{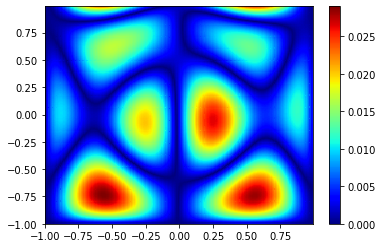}}
		{ $| u_{\vTheta} - u | $}
		&
		\subf{\includegraphics[width=50mm]{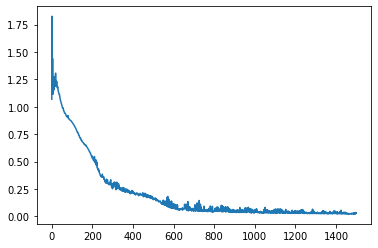}}
		{ $\|u_{\vTheta} - u\|_{L^2(\Omega)}$ vs. iterations.}
		\\
		\hline
	\end{tabular}
	\caption{Top left: computational solution $u_{\vTheta}$ to \eqref{eq:example1} in case $k = 1$ and $d = 2$. In the computation, we used a learning rate $\ell = 0.005$, with 2,000 collocation points, 1,500 optimization steps, and 3603 degrees of freedom (including auxiliary functions). We used one-layer networks both for the main and auxiliary functions. The panel in bottom left exhibits the pointwise discrepancy $|u-u_\vTheta|$. We also report the evolution of the loss function (top right) and the $L^2$ error (bottom right).}
	\label{fig:dim2}
\end{figure}

\begin{figure}[h]
	\centering
	\begin{tabular}{|c|c|c|c|}
		\hline
		\subf{\includegraphics[width=50mm]{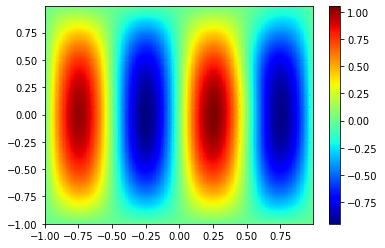}}
		{$ u_{\vTheta}$}
		&
		\subf{\includegraphics[width=50mm]{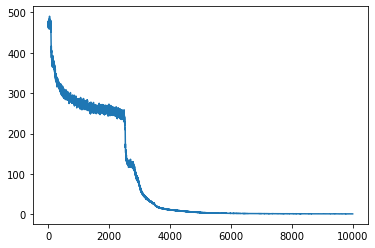}}
        {Loss function vs. iterations.}
		\\
		\hline
		\subf{\includegraphics[width=50mm]{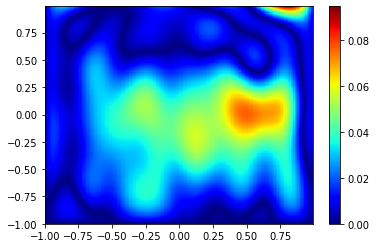}}
		{ $| u_{\vTheta} - u | $}
		&
		\subf{\includegraphics[width=50mm]{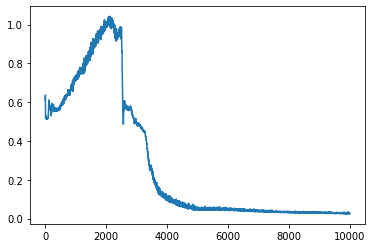}}
		{ $\|u_{\vTheta} - u\|_{L^2(\Omega)}$ vs. iterations.}
		\\
		\hline
	\end{tabular}
	\caption{Computational solution $u_{\vTheta}$ (top left), evolution of the loss function (top right), pointwise error (bottom left), and evolution of the $L^2$-error (bottom right) for \eqref{eq:example1} with $k = 2$ and $d = 2$. We employed five-layer networks for the main functions and three-layer networks for the auxiliary functions. We used an initial learning rate $\ell = 0.005$, with 5,000 collocation points, 10,000 optimization steps, and 2901 degrees of freedom. We halved the learning rate every 2,500 optimization steps.}
	\label{fig:dim2_k2}
\end{figure}

\begin{figure}[h]
	\centering
	\begin{tabular}{|c|c|c|c|}
		\hline
		\subf{\includegraphics[width=50mm]{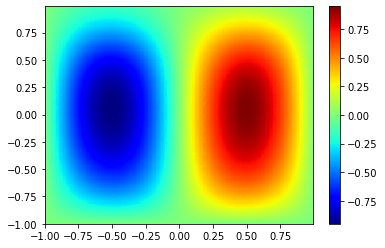}}
		{$ \restr{u_{\vTheta}}{\{x_3,\ldots,x_5 = 0.5\}}$}
		&
		\subf{\includegraphics[width=50mm]{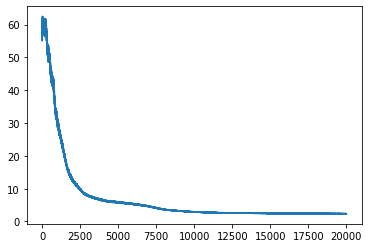}}
        {Loss function vs. iterations.}
		\\
		\hline
		\subf{\includegraphics[width=50mm]{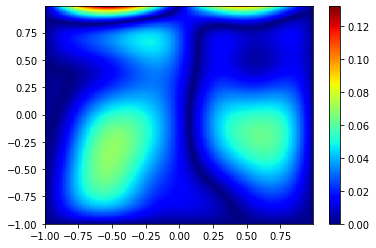}}
		{ $|\restr{u_{\vTheta}}{\{x_3,\ldots,x_5 = 0.5\}} - 
		\restr{u}{\{x_3,\ldots,x_5 = 0.5\}}|$}
		&
		\subf{\includegraphics[width=50mm]{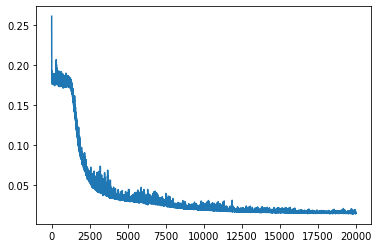}}
		{ MSE vs. iterations.}
		\\
		\hline
	\end{tabular}
	\caption{Slice of the solution $u_{\vTheta}$ (top left), evolution of the loss function (top right), pointwise error (bottom left), and evolution of the Mean Squared Error (MSE) (bottom right) for \eqref{eq:example1} with $k = 1$ and $d = 5$. We employed five-layer networks for the main functions and three-layer networks for the auxiliary functions. We used an initial learning rate $\ell = 0.005$, with 12,000 collocation points, 20,000 optimization steps, and 5656 degrees of freedom. We halved the learning rate every 4,000 optimization steps. We estimated the MSE by using 5,000 random points in $\Omega$ (re-sampled at every step}
	\label{fig:altadim}
\end{figure}
\end{example}

\begin{example}[singularly perturbed problem] \label{ex:singularly-perturbed}
Let $\eps > 0$, $\Omega = (0,1)^2$, $\vb = (-1+2\varepsilon,-1+2\varepsilon)$, $c = 2(1-\varepsilon)$, and the function $f \colon \Omega \to \R$,
\[
f(x,y) = - \left[  x - \left( \frac{1-e^{-x/\eps}}{1-e^{-1/\eps}}\right) + y - \left( \frac{1-e^{-y/\eps}}{1-e^{-1/\eps}}\right) \right]e^{x+y}.
\]
We consider the singularly perturbed problem: find $u \colon \Omega \to \R$ such that
\begin{equation} \label{eq:example2}
\left\lbrace \begin{aligned}
	-\varepsilon \Delta u + \vb\cdot \nabla u +cu = f \quad & \mbox{in } \Omega,\\
	u = 0 \quad & \mbox{on } \partial\Omega.
	\end{aligned} \right. 
\end{equation}
The exact solution to \eqref{eq:example2} is
\[
u(x,y) =  \left( x - \frac{1-e^{-x/\eps}}{1-e^{-1/\eps}}\right) \left( y - \frac{1-e^{-y/\eps}}{1-e^{-1/\eps}}\right) e^{x+y}.
\]
Figure \ref{fig:sing_pert} exhibits our computed solutions for this example with $\eps = 0.05$. In that case, we observed a fast convergence towards the solution, reaching $L_N (\vTheta) = 0.0112$, as well as a good adaptation of the discrete solution to the boundary layers. 

\begin{figure}[h]
	\centering
	\begin{tabular}{|c|c|c|c|}
		\hline
		\subf{\includegraphics[width=45mm]{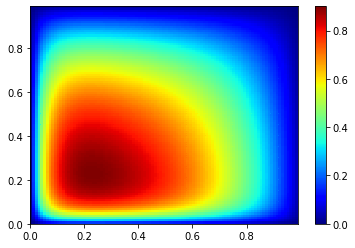}}
		{$u_{\vTheta}(x)$}
		&
		\subf{\includegraphics[width=45mm]{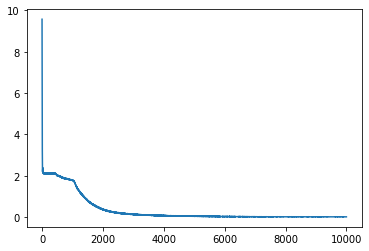}}
        {Loss function vs. iterations.}
		\\
		\hline
		\subf{\includegraphics[width=45mm]{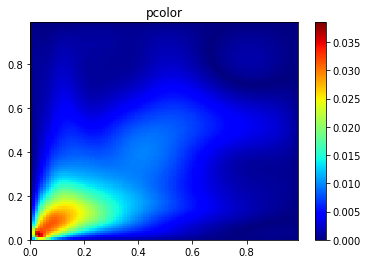}}
		{$|u_{\vTheta}(x) - u(x)|$}
		&
		\subf{\includegraphics[width=45mm]{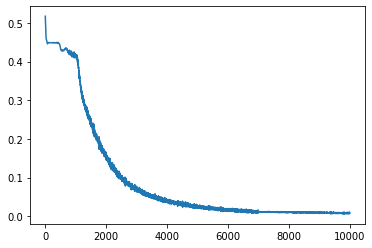}}
		{$\|u_{\vTheta} - u\|_{L^2(\Omega)}$ vs. iterations.}
		\\
		\hline
	\end{tabular}
	\caption{
	Computational solution $u_{\vTheta}$ (top left), evolution of the loss function (top right), pointwise error (bottom left), and evolution of the $L^2$-error (bottom right) for \eqref{eq:example2} with $\eps=0.05$. We used an initial learning rate $\ell = 0.005$, with 5,000 collocation points, 10,000 optimization steps, and 3543 degrees of freedom. Auxiliary functions have been approximated exactly.
	}
	\label{fig:sing_pert}
\end{figure}
\end{example}

\section{Concluding remarks} \label{sec:conclusion}

In this work, we have proposed a First-Order System Least Squares (FOSLS) method based on deep learning for numerically solving second-order elliptic PDEs. This method is meshless, which is naturally advantageous for high-dimensional problems, but as a consequence implies that we cannot compute the loss functions exactly. Taking into account this practical issue, we proved the almost sure convergence of the neural network minimizers towards the PDE solutions. We furthermore extended the theoretical framework to incorporate other methods based on Monte Carlo quadrature.

\begin{remark}[almost-everywhere solutions]\label{rem:regularization}
The convergence proofs in Sections \ref{sec:analysis} and \ref{sec:general} are based on the use of regularized versions of the cost functionals and their discretizations. Regularization consists in restricting the size of the parameters, namely, imposing that $| \vTheta | < R$ for certain $R < \infty$. This ensures that any neural network function with large derivatives is penalized, thereby preventing minimizers from approximating non-smooth functions.

Far from being an artificial condition of the proof, regularization mechanisms of this kind are necessary in the implementation to avoid convergence towards functions that satisfy the PDE almost everywhere but are not weak solutions of the target problem. To illustrate this point, consider the following example, which is just \eqref{eq:FOS} in a simplified setting: seek $u,\phi: (0,1) \to \R$ such that
\begin{equation}\label{eq:remark}
\left\lbrace
\begin{aligned}
\phi - u' & = 0 & \mbox{in } (0,1), \\
\phi'&  = 0 & \mbox{in } (0,1), \\
u(0) & = 0,  \\
u(1) & = 1.
\end{aligned}
\right.
\end{equation}
Naturally, the unique minimizer of the least-squares functional (cf. \eqref{eq:LS-loss})
\[
\calL (u, \phi) := \|  \phi - u' \|_{L^2(\Omega)}^2 + \| \phi' \|_{L^2(\Omega)}^2
\]
in the corresponding admissible set $\calA = \{ (u,\phi) \in [H^1(\Omega)]^2 \colon u(0) = 0, \ u(1) = 1 \}$ is $u(x) = x$ and $\phi(x) = 1$. Let $\delta \in (0,1/2)$ be a small number, and consider the functions
\begin{equation} \label{eq:def-u-phi}
u_\delta(x) = 
\left\lbrace
\begin{aligned}
& 0 & \mbox{in } (0,1/2-\delta) \\
& \frac{x-1/2+\delta}{2\delta} & \mbox{in } (1/2-\delta,1/2+\delta) \\
& 1 & \mbox{in } (1/2+\delta,1)
\end{aligned}
\right.,
\qquad 
\phi_\delta(x) =
\left\lbrace
\begin{aligned}
& 0 & \mbox{in } (0,1/2-\delta) \\
& \frac1{2\delta} & \mbox{in } (1/2-\delta,1/2+\delta) \\
& 0 & \mbox{in } (1/2+\delta,1)
\end{aligned}
\right. .
\end{equation}
We notice $\phi_\delta = u'_\delta$ a.e. in $(0,1)$ and $\calL (u_\delta, \phi_\delta) = 0$, although $(u_\delta, \phi_\delta) \notin \calA$, because $\phi_\delta$ is not an $H^1$ function.

If we utilize the discrete functional \eqref{eq:discrete_cost} with collocation points, and none of these points lies in the interval $(1/2-\delta, 1/2+\delta)$, then for these two functions we would have
\[
\calL_N (u_\delta, \phi_\delta) = 0.
\]
We remark that, independently of the number of collocation points $N$, one can always take $\delta > 0$ sufficiently small such that the probability of none of the sampling points lies in $(1/2-\delta, 1/2+\delta)$ is significant. Therefore, if our neural network is capable of producing functions $(u_{\vTheta}, \phi_{\vTheta})$ approximating $(u_\delta, \phi_\delta)$ in \eqref{eq:def-u-phi} (cf. Figure \ref{fig:remark}), then during the optimization process the descent algorithm may choose to approximate the pair $(u, \phi) = (\chi_{(1/2,1)}, 0)$. This function satisfies the differential equations in \eqref{eq:remark} almost everywhere, but is not a significant solution. The issue of approximating bad solutions of this kind is mitigated by applying classic regularization techniques that penalize large parameters, because $|\vTheta|$ must be large in order to $u_\vTheta '$ be large at some portion of the domain.
\begin{figure}[h]
    \centering
    \includegraphics[scale=0.4]{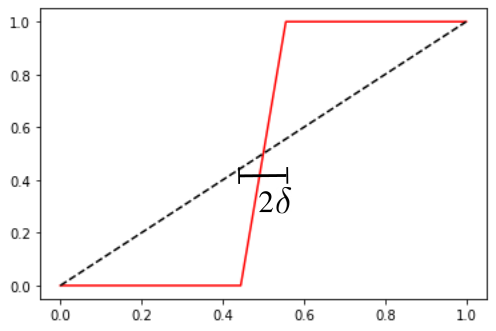}
    \caption{In red the function $u_{\delta}$ defined in \ref{eq:def-u-phi}. On dashed lines the solution of problem \eqref{eq:remark}.}
    \label{fig:remark}
\end{figure}

This difficulty extends to all methods based on the minimization of cost functionals similar to \eqref{eq:discrete_cost}, such as DGM \cite{DGM18} or DRM \cite{E18}. The issue stems from the fact that the functional \eqref{eq:discrete_cost} is unable to distinguish between regular solutions (belonging to a suitable Sobolev space) from any other functions that satisfy the equation almost everywhere. As far as we know, this problem has not been addressed in the literature, and the question of how to develop suitable regularization techniques for these approaches remains open. 
\end{remark}

\begin{remark}[approximation of non-smooth solutions] \label{rem:non-smooth}
There are, however, problems in which the solution presents large gradients in regions of the domain. One can typically think of singularly perturbed problems, such as \eqref{eq:example2}, or singularities arising due to poor boundary regularity, such as for the Poisson problem on an $L$-shaped domain. In those problems, regularization can limit the approximation capabilities of the algorithm. 

For algebraic boundary singularities, if the boundary conditions are imposed in a strong fashion, as discussed in Section \ref{sec:boundary-data}, one could aim to modify the rate at which the corresponding auxiliary function $\dD$ or $\dN$ decreases to zero near the singularity. This could potentially avoid $v_{\vTheta}$ having to approximate a singular function and lead to a faster convergence. Nevertheless, this requires an a priori knowledge about the location and behavior of the singularities of the solution, that is not available in general. 
We emphasize that the theory we developed in Section \ref{sec:general} does not make any regularity assumption on the PDE, and therefore includes the case of non-smooth solutions.
\end{remark}

\section*{Acknowledgements}
The authors thank Prof. Michael Karkulik and Roberto Gonz\'alez for their insightful comments on a previous version of this manuscript.

\bibliographystyle{abbrv}
\bibliography{FOSLS.bib}

\begin{thebibliography}{10}

\bibitem{Arora18}
R.~Arora, A.~Basu, P.~Mianjy, and A.~Mukherjee.
\newblock Understanding deep neural networks with rectified linear units.
\newblock In {\em International Conference on Learning Representations}, 2018.

\bibitem{Barron93}
A.~Barron.
\newblock Universal approximation bounds for superpositions of a sigmoidal
  function.
\newblock {\em IEEE Transactions on Information theory}, 39(3):930--945, 1993.

\bibitem{Berg18}
J.~Berg and K.~Nystr{\"o}m.
\newblock A unified deep artificial neural network approach to partial
  differential equations in complex geometries.
\newblock {\em Neurocomputing}, 317:28--41, 2018.

\bibitem{braides2006}
A.~Braides.
\newblock A handbook of {$\Gamma$}-convergence.
\newblock In {\em Handbook of Differential Equations: stationary partial
  differential equations}, volume~3, pages 101--213. Elsevier, 2006.

\bibitem{Cai20}
Z.~Cai, J.~Chen, M.~Liu, and X.~Liu.
\newblock Deep least-squares methods: An unsupervised learning-based numerical
  method for solving elliptic {PDEs}.
\newblock {\em Journal of Computational Physics}, 420:109707, 2020.

\bibitem{Cai94}
Z.~Cai, R.~Lazarov, T.~Manteuffel, and S.~McCormick.
\newblock First-order system least squares for second-order partial
  differential equations: Part {I}.
\newblock {\em SIAM Journal on Numerical Analysis}, 31(6):1785--1799, 1994.

\bibitem{Cybenko89}
G.~Cybenko.
\newblock Approximation by superpositions of a sigmoidal function.
\newblock {\em Mathematics of control, signals and systems}, 2(4):303--314,
  1989.

\bibitem{weinan2022some}
W.~E and S.~Wojtowytsch.
\newblock {Some observations on high-dimensional partial differential equations
  with Barron data}.
\newblock In {\em Mathematical and Scientific Machine Learning}, pages
  253--269. PMLR, 2022.

\bibitem{E18}
W.~E and B.~Yu.
\newblock {The deep Ritz method: a deep learning-based numerical algorithm for
  solving variational problems}.
\newblock {\em Communications in Mathematics and Statistics}, 6(1):1--12, 2018.

\bibitem{He20}
C.~He, X.~Hu, and L.~Mu.
\newblock A mesh-free method using piecewise deep neural network for elliptic
  interface problems.
\newblock {\em Journal of Computational and Applied Mathematics}, 412:114358,
  2022.

\bibitem{he2022relu}
J.~He, L.~Li, and J.~Xu.
\newblock Relu deep neural networks from the hierarchical basis perspective.
\newblock {\em Computers \& Mathematics with Applications}, 120:105--114, 2022.

\bibitem{He_etal20}
J.~He, L.~Li, J.~Xu, and C.~Zheng.
\newblock Relu deep neural networks and linear finite elements.
\newblock {\em J. Comput. Math.}, 38(3):502--527, 2020.

\bibitem{hornik1991}
K.~Hornik.
\newblock Approximation capabilities of multilayer feedforward networks.
\newblock {\em Neural networks}, 4(2):251--257, 1991.

\bibitem{ADAM}
D.~Kingma and J.~Ba.
\newblock Adam: A method for stochastic optimization.
\newblock In {\em In Proceedings of the 3rd InternationalConference for
  Learning Representations---ICLR}, pages 7--9, San Diego, CA, 2015.

\bibitem{lagaris1998}
I.~E. Lagaris, A.~Likas, and D.~I. Fotiadis.
\newblock Artificial neural networks for solving ordinary and partial
  differential equations.
\newblock {\em IEEE transactions on neural networks}, 9(5):987--1000, 1998.

\bibitem{lagaris2000}
I.~E. Lagaris, A.~C. Likas, and D.~G. Papageorgiou.
\newblock Neural-network methods for boundary value problems with irregular
  boundaries.
\newblock {\em IEEE Transactions on Neural Networks}, 11(5):1041--1049, 2000.

\bibitem{lee1990}
H.~Lee and I.~S. Kang.
\newblock Neural algorithm for solving differential equations.
\newblock {\em Journal of Computational Physics}, 91(1):110--131, 1990.

\bibitem{liu2}
M.~Liu and Z.~Cai.
\newblock {Adaptive two-layer ReLU neural network: II. Ritz approximation to
  elliptic PDEs}.
\newblock {\em Computers \& Mathematics with Applications}, 113:103--116, 2022.

\bibitem{liu1}
M.~Liu, Z.~Cai, and J.~Chen.
\newblock {Adaptive two-layer ReLU neural network: I. best least-squares
  approximation}.
\newblock {\em Computers \& Mathematics with Applications}, 113:34--44, 2022.

\bibitem{Lyu20}
L.~Lyu, Z.~Zhang, M.~Chen, and J.~Chen.
\newblock Mim: A deep mixed residual method for solving high-order partial
  differential equations.
\newblock {\em Journal of Computational Physics}, 452:110930, 2022.

\bibitem{malek2006}
A.~Malek and R.~S. Beidokhti.
\newblock Numerical solution for high order differential equations using a
  hybrid neural network---optimization method.
\newblock {\em Applied Mathematics and Computation}, 183(1):260--271, 2006.

\bibitem{PINN}
M.~Raissi, P.~Perdikaris, and G.~E. Karniadakis.
\newblock Physics-informed neural networks: A deep learning framework for
  solving forward and inverse problems involving nonlinear partial differential
  equations.
\newblock {\em Journal of Computational physics}, 378:686--707, 2019.

\bibitem{Sheng20}
H.~Sheng and C.~Yang.
\newblock {PFNN:} a penalty-free neural network method for solving a class of
  second-order boundary-value problems on complex geometries.
\newblock {\em Journal of Computational Physics}, 428:110085, 2021.

\bibitem{shin2020}
Y.~Shin, J.~Darbon, and G.~E. Karniadakis.
\newblock On the convergence of physics informed neural networks for linear
  second-order elliptic and parabolic type pdes.
\newblock {\em Communications in Computational Physics}, 28(5):2042--2074,
  2020.

\bibitem{hong2021priori}
J.~Siegel, Q.~Hong, X.~Jin, W.~Hao, and J.~Xu.
\newblock {A priori analysis of stable neural network solutions to numerical
  PDEs}.
\newblock {\em arXiv preprint arXiv:2107.04466}, 2022.

\bibitem{siegel2020high}
J.~W. Siegel and J.~Xu.
\newblock High-order approximation rates for neural networks with {R}e{LU}$^k$
  activation functions.
\newblock {\em arXiv preprint arXiv:2012.07205}, 2020.

\bibitem{siegel2021sharp}
J.~W. Siegel and J.~Xu.
\newblock Sharp lower bounds on the approximation rate of shallow neural
  networks.
\newblock {\em arXiv preprint arXiv:2106.14997}, 2021.

\bibitem{DGM18}
J.~Sirignano and K.~Spiliopoulos.
\newblock {DGM: A deep learning algorithm for solving partial differential
  equations}.
\newblock {\em Journal of Computational Physics}, 375:1339--1364, 2018.

\bibitem{Wang20}
Z.~Wang and Z.~Zhang.
\newblock A mesh-free method for interface problems using the deep learning
  approach.
\newblock {\em Journal of Computational Physics}, 400:108963, 2020.

\bibitem{wojtowytsch2020can}
S.~Wojtowytsch and W.~E.
\newblock {Can shallow neural networks beat the curse of dimensionality? A mean
  field training perspective}.
\newblock {\em IEEE Transactions on Artificial Intelligence}, 1(2):121--129,
  2020.

\bibitem{FNM}
J.~Xu.
\newblock Finite neuron method and convergence analysis.
\newblock {\em Communications in Computational Physics}, 28:1707--1745, 2020.

\bibitem{Yarotsky17}
D.~Yarotsky.
\newblock Error bounds for approximations with deep relu networks.
\newblock {\em Neural Networks}, 94:103--114, 2017.

\bibitem{Zang20}
Y.~Zang, G.~Bao, X.~Ye, and H.~Zhou.
\newblock Weak adversarial networks for high-dimensional partial differential
  equations.
\newblock {\em Journal of Computational Physics}, page 109409, 2020.

\bibitem{zerbinati2022pinns}
U.~Zerbinati.
\newblock {PINNs and GaLS:} a priori error estimates for shallow physics
  informed neural networks applied to elliptic problems.
\newblock {\em IFAC-PapersOnLine}, 55(20):61--66, 2022.
\newblock 10th Vienna International Conference on Mathematical Modelling
  MATHMOD 2022.

\end{thebibliography}
\end{document}